\documentclass[11pt,reqno]{amsart}
\usepackage{amsmath,amsfonts,amssymb,amsthm,booktabs,epsfig,enumitem,listings,mathrsfs,setspace}
\usepackage[margin=1.375in]{geometry}
\usepackage{color, bm}
\usepackage{comment}
\usepackage[colorlinks=true, pdfstartview=FitV, linkcolor=blue, citecolor=blue, urlcolor=blue, bookmarksdepth=2]{hyperref}
\usepackage[all]{hypcap}

\usepackage{tikz}
\usetikzlibrary{matrix}
\tikzset{tab/.style={matrix of math nodes,column sep=-.35, row sep=-.35,text height=7pt,text width=7pt,align=center,inner sep=2,font=\footnotesize}}
\tikzset{dynkdot/.style={circle,draw,scale=.45}}
\tikzset{dot/.style={circle,draw,fill,scale=.45}}
\newcommand{\tikzmark}[2]{\tikz[overlay,remember picture,baseline] \node [anchor=base] (#1) {$#2$};}
\newcommand{\DrawLine}[3][]{%
  \begin{tikzpicture}[overlay,remember picture]
    \draw[#1] (#2.210) -- (#3.30);
  \end{tikzpicture}
}

\makeatletter
\def\@tocline#1#2#3#4#5#6#7{\relax
  \ifnum #1>\c@tocdepth % then omit
  \else
    \par \addpenalty\@secpenalty\addvspace{#2}%
    \begingroup \hyphenpenalty\@M
    \@ifempty{#4}{%
      \@tempdima\csname r@tocindent\number#1\endcsname\relax
    }{%
      \@tempdima#4\relax
    }%
    \parindent\z@ \leftskip#3\relax \advance\leftskip\@tempdima\relax
    \rightskip\@pnumwidth plus4em \parfillskip-\@pnumwidth
    #5\leavevmode\hskip-\@tempdima
      \ifcase #1
       \or\or \hskip 1em \or \hskip 2em \else \hskip 3em \fi%
      #6\nobreak\relax
    \dotfill\hbox to\@pnumwidth{\@tocpagenum{#7}}\par
    \nobreak
    \endgroup
  \fi}
\makeatother

\numberwithin{equation}{section}
\numberwithin{figure}{section}
\numberwithin{table}{section}
\newtheorem{Theorem}[equation]{Theorem}
\newtheorem{ThmDef}[equation]{Theorem/Definition}
\newtheorem{Proposition}[equation]{Proposition} 
\newtheorem{Lemma}[equation]{Lemma}

\newtheorem{Corollary}[equation]{Corollary}

\theoremstyle{definition}

\newtheorem{Definition}[equation]{Definition}
\newtheorem{Example}[equation]{Example}
\newtheorem{Remark}[equation]{Remark}

\newenvironment{acknowledgements}
{\bigskip\noindent\footnotesize\textbf{Acknowledgements.} }
{\medskip}

\newcommand{\arxiv}[1]{\href{http://arxiv.org/abs/#1}{\tt arXiv:\nolinkurl{#1}}}

\newcommand{\bc}{\mathbf{C}}

\newcommand{\bq}{\mathbf{Q}}

\newcommand{\bz}{\mathbf{Z}}
\newcommand{\CB}{\mathscr{B}}
\newcommand{\cc}{{\bm c}}

\newcommand{\defn}[1]{{\color{blue}\it#1}}
\newcommand{\ghost}{\boldsymbol{0}}
\newcommand{\ii}{{\bf i}}

\newcommand{\stack}[1]{\begin{smallmatrix}#1\end{smallmatrix}}

\newcommand{\wt}{{\rm wt}}

\usepackage{listings}
\lstdefinelanguage{Sage}[]{Python}
{morekeywords={False,True},sensitive=true}
\begin{document}

\title[Combinatorial descriptions of PBW crystals]{Combinatorial descriptions of the crystal structure on certain PBW bases}

\author{Ben Salisbury}
\thanks{B.S.\ was partially supported by CMU Early Career grant \#C62847 and by Simons Foundation grant \#429950}
\address{Department of Mathematics, Central Michigan University, Mount Pleasant, MI}
\email{ben.salisbury@cmich.edu}
\urladdr{http://people.cst.cmich.edu/salis1bt/}
\author{Adam Schultze}
\thanks{A.S.\ and P.T.\ were partially supported by NSF grant DMS-1265555}
\address{Department of Mathematics and Statistics, University at Albany, Albany, NY}
\email{aschultze@albany.edu}
\author{Peter Tingley}
\address{Department of Mathematics and Statistics, Loyola University, Chicago, IL}
\email{ptingley@luc.edu}
\urladdr{http://webpages.math.luc.edu/~ptingley/}

\begin{abstract}
Using the theory of PBW bases, one can realize the crystal $B(\infty)$ for any semisimple Lie algebra over $\mathbf{C}$ using Kostant partitions as the underlying set. In fact there are many such realizations, one for each reduced expression for the longest element of the Weyl group. There is an algorithm to calculate the actions of the crystal operators, but it can be quite complicated. Here we show that, for certain reduced expressions, the crystal operators can also be described by a much simpler bracketing rule. We give conditions describing these reduced expressions, and show that there is at least one example in every type except possibly $E_8$, $F_4$ and $G_2$. We then discuss some examples. 
\end{abstract}

\maketitle

\tableofcontents
 
\section{Introduction}

The crystal $B(\infty)$ of a semisimple Lie algebra $\mathfrak{g}$ over $\bc$ is a combinatorial object that contains a lot of information about $\mathfrak{g}$ and its finite-dimensional representations. 
Lusztig's early construction of canonical basis can be interpreted as giving a number of parameterizations of $B(\infty)$, one for each reduced expression for the longest element $w_0$ in the Weyl group (see \cite[Chapters 41 and 42]{L10} or  \cite{tingley}). In each of these realizations at least one of the crystal operators is very simple, but others may be complicated. However, Lusztig explicitly describes how the realizations are related for reduced expressions that differ by a braid move (see also Berenstein and Zelevinsky \cite{BZ97} for the non-simply-laced cases). This gives a way to realize the whole crystal: an element is a PBW monomial with respect to some chosen reduced expression. To apply a crystal operator, modify the element via a sequence of braid moves until that operator is simple, then apply the operator, then modify it back. 

This procedure is algorithmic, but can be complicated. In type $A_n$ there is a simpler realization, using multisegments, where the crystal operators are given by a bracketing rule. As discussed in \cite{CT15}, this is naturally identified with Lusztig's crystal structure for the reduced expression 
\[
w_0= (s_1 s_2 s_3 \cdots s_n)(s_1 \cdots s_{n-1}) \cdots (s_1 s_2) s_1.
\]
Our main result is to generalize this by giving conditions on a reduced expression that ensure
Lusztig's crystal structure 
is given by a similar rule, and to describe the resulting structure. For these words the crystal operators can be understood by combining many rank two calculations, and this combining is controlled by a bracketing procedure. 
There is at least one such reduced expression in every type except possibly $E_8$ and $F_4$, and we give detailed examples in the classical types. We do not consider type $G_2$ simply because the rank 2 calculation is more difficult.
The method in \cite{CT15} does not generalize easily outside of type $A_n$, and the proof here is quite different. 

Reineke \cite{reineke} has also given explicit rules for calculating crystal operators on certain PBW monomials/Kostant partitions, which he does by identifying a Kostant partition with an isomorphism class of quiver representation.  However, our method works in different generality: Reineke requires the reduced expression to be adapted to some orientation of the Dynkin diagram, while we require a ``simply braided" reduced expression (see Definition \ref{def:sa}). For example, the reduced expression we use for type $D_n$ in \S\ref{ss:tdn} is not adapted to any orientation. Kwon \cite{kwon1} recently showed that Reineke's structure does come from a bracketing rules in some cases (type $A$ when the orientation has a single sink), and it would be interesting to compare that to the construction here.

\section{Background}

Let $\mathfrak{g}$ be a simple, finite-dimensional Lie algebra over $\bc$.
Let $I$ be the index set of $\mathfrak{g}$, $A = (a_{ij})$ the Cartan matrix, $\{ \alpha_i \}_{i \in I}$ the positive simple roots, $\{\alpha_i^\vee\}_{i\in I}$ the simple coroots, $\Phi$ the set of roots,  $\Phi^+ \subset \Phi$ the positive roots, $P$ the weight lattice, $P^\vee$ the dual weight lattice, $W$ the Weyl group with longest element $w_0$, and $\{ s_i \}_{i \in I}$ the generating simple reflections. If $\mathfrak{g}$ is not simply-laced, let $\mathrm{diag}(d_i : i\in I)$ be the symmetrizing matrix of $A$.  Let $(-|-)$ denote the symmetric, $W$-invariant bilinear form on $P$ satisfying $(\alpha_i|\alpha_j) = d_ia_{ij}$ and let $\langle-,-\rangle \colon P^\vee \times P \longrightarrow \bz$ be the canonical pairing. Denote the set of all reduced expressions for the longest element of the Weyl group by $R(w_0)$.  Elements of $R(w_0)$ may also be referred to as reduced long words, or simply as reduced words.
Enumerate the Dynkin diagram of $\mathfrak{g}$ following Bourbaki \cite{bourbaki} (see Figure \ref{fig:diagrams}).

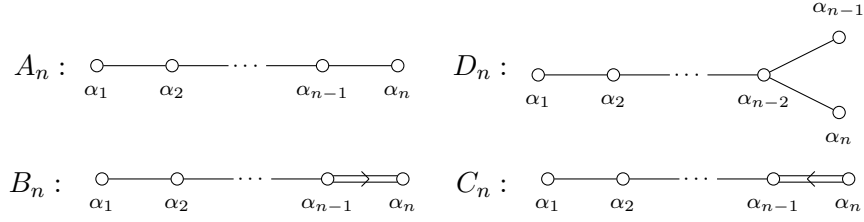
\begin{figure}
\[
\begin{array}{cc}
A_n: 
\begin{tikzpicture}[baseline=-3.5,scale=1,font=\scriptsize]
\foreach \x in {1,2,4,5}
{\node[dynkdot] (\x) at (\x,0) {};}
\node[label={below:$\alpha_1$}] at (1,0) {}; 
\node[label={below:$\alpha_2$}] at (2,0) {}; 
\node[label={below:$\alpha_{n-1}$}] at (4,0) {}; 
\node[label={below:$\alpha_n$}] at (5,0) {}; 
\node at (3,0) {$\cdots$};
\draw[-] (1) -- (2);
\draw[-] (2) -- (2.75,0);
\draw[-] (3.25,0) -- (4);
\draw[-] (4) -- (5);
\end{tikzpicture} 
&
D_n:
\begin{tikzpicture}[baseline,scale=1,font=\scriptsize]
\foreach \x in {1,2,4}
{\node[circle,draw,scale=.45] (\x) at (\x,0) {};}
\node[label={below:$\alpha_1$}] at (1,0) {}; 
\node[label={below:$\alpha_2$}] at (2,0) {}; 
\node[label={below:$\alpha_{n-2}$}] at (4,0) {}; 
\node[label={above:$\alpha_{n-1}$}] at (5,.5) {}; 
\node[label={below:$\alpha_{n}$}] at (5,-.5) {};
\node at (3,0) {$\cdots$};
\node[circle,draw,scale=.45] (5) at (5,.5) {};
\node[circle,draw,scale=.45] (6) at (5,-.5) {};
\draw[-] (1) -- (2);
\draw[-] (2) -- (2.75,0);
\draw[-] (3.25,0) -- (4);
\draw[-] (4) -- (5);
\draw[-] (4) -- (6);
\end{tikzpicture}
\\[30pt]
B_n: \ 
\begin{tikzpicture}[baseline=-5,scale=1,font=\scriptsize]
\foreach \x in {1,2,4,5}
{\node[dynkdot] (\x) at (\x,0) {};}
\node[label={below:$\alpha_1$}] at (1,0) {}; 
\node[label={below:$\alpha_2$}] at (2,0) {}; 
\node[label={below:$\alpha_{n-1}$}] at (4,0) {}; 
\node[label={below:$\alpha_n$}] at (5,0) {}; 
\node at (3,0) {$\cdots$};
\draw[-] (1.east) -- (2.west);
\draw[-] (2.east) -- (2.75,0);
\draw[-] (3.25,0) -- (4.west);
\draw[-] (4.30) -- (5.150);
\draw[-] (4.330) -- (5.210);
\draw[-] (4.55,0) -- (4.45,.1);
\draw[-] (4.55,0) -- (4.45,-.1);
\end{tikzpicture}
&
C_n: \
\begin{tikzpicture}[baseline=-5,scale=1,font=\scriptsize]
\foreach \x in {1,2,4,5}
{\node[dynkdot] (\x) at (\x,0) {};}
\node[label={below:$\alpha_1$}] at (1,0) {}; 
\node[label={below:$\alpha_2$}] at (2,0) {}; 
\node[label={below:$\alpha_{n-1}$}] at (4,0) {}; 
\node[label={below:$\alpha_n$}] at (5,0) {}; 
\node at (3,0) {$\cdots$};
\draw[-] (1.east) -- (2.west);
\draw[-] (2.east) -- (2.75,0);
\draw[-] (3.25,0) -- (4.west);
\draw[-] (4.30) -- (5.150);
\draw[-] (4.330) -- (5.210);
\draw[-] (4.45,0) -- (4.55,.1);
\draw[-] (4.45,0) -- (4.55,-.1);
\end{tikzpicture}
\end{array}
\]
\caption{\label{fig:diagrams} Dynkin diagrams of classical type following Bourbaki \cite{bourbaki}.}
\end{figure}

Let $U_q(\mathfrak{g})$ be the quantized universal enveloping algebra of $\mathfrak{g}$, which is a $\bq(q)$-algebra generated by $E_i$, $F_i$, and $q^h$, for $i\in I$ and $h\in P^\vee$, subject to certain relations (see, for example, \cite{HK02}). Let $U_q^-(\mathfrak{g})$ be the subalgebra generated by the $F_i$'s. The star involution is the involutive $\bq(q)$-algebra antiautomorphism $*\colon U_q(\mathfrak{g}) \longrightarrow U_q(\mathfrak{g})$ defined by
\[
E_i^* = E_i, \ \ \ 
F_i^* = F_i, \ \ \ 
%q^* = q, \ \ \ 
(q^h)^* = q^{-h}.
\]

\subsection{Crystals}

Let $e_{i}$, $f_{i}$ be the Kashiwara operators on $U_q^-(\mathfrak{g})$ defined in \cite{K91}. Let $\mathcal{A} \subset \bq(q)$ be the subring of functions regular at $q=0$ and define $L(\infty)$ to be the $\mathcal{A}$-lattice spanned by
\[
S= \{  f_{i_1}  f_{i_2} \cdots  f_{i_t} \cdot 1 \in U_q^-(\mathfrak{g}) : t\ge 0, \ i_k \in I \}.
\]
Let $e_i^* = *\circ e_i \circ *$ and $f_i^* = *\circ f_i \circ *$ be the operators twisted by the $*$-involution.

\begin{ThmDef}[\cite{K91}] \hfill 
\begin{enumerate}
\item Let $\pi \colon L(\infty) \longrightarrow L(\infty)/qL(\infty)$ be the natural projection and set $B(\infty) = \pi(S)$. Then $B(\infty)$ is a $\bq$-basis of $L(\infty)/qL(\infty)$.
\item For each $i\in I$ the operators $e_i$ and $f_i$ act on $L(\infty)/qL(\infty)$. Moreover, $e_i\bigl(B(\infty)\bigr) = B(\infty)\sqcup \{\ghost\}$ and $f_i\bigl(B(\infty)\bigr) \subset B(\infty)$.
\item The involution $*$ preserves $L(\infty)$ and $B(\infty)$. Hence $e_i^*$ and $f_i^*$ act on $B(\infty)$. 
\end{enumerate}
\end{ThmDef}
 
For $i\in I$ and $b\in B(\infty)$, define
\begin{align*}
\varepsilon_i(b) = \max\{ k \in \bz_{\ge0} : e_i^kb \neq \ghost \}, \qquad
\varepsilon_i^*(b) = \max\{ k \in \bz_{\ge0} : (e_i^*)^kb \neq \ghost \}.
\end{align*}
Consider the weight map $\wt\colon B(\infty) \longrightarrow P$ defined by 
\[
\wt(f_{i_1} f_{i_2} \cdots f_{i_t} \cdot 1 ) = -\alpha_{i_1}-\alpha_{i_2}-\cdots-\alpha_{i_t}.
\]
The next proposition follows from \cite[Prop.~3.2.3]{KS97} (see also \cite[Prop.~1.4]{TW12}).  

\begin{Proposition}
\label{Prop:TW}
For all $b\in B(\infty)$ and all $i\neq j$ in $I$, we have
\begin{enumerate}
\item $f_i(b), f_i^*(b) \neq \ghost$,
\item $f_i^*f_j(b) = f_jf_i^*(b)$, 
\item $\varepsilon_i(b) + \varepsilon_i^*(b) + \langle\alpha_i^\vee,\wt(b)\rangle \ge 0$,
\item $\varepsilon_i(b) + \varepsilon_i^*(b) + \langle\alpha_i^\vee,\wt(b)\rangle = 0$ implies $f_i(b) = f_i^*(b)$,
\item $\varepsilon_i(b) + \varepsilon_i^*(b) + \langle\alpha_i^\vee,\wt(b)\rangle \ge 1$ implies $\varepsilon_i^*\bigl(f_i(b)\bigr) = \varepsilon_i^*(b)$ and $\varepsilon_i\bigl(f_i^*(b)\bigr) = \varepsilon_i(b)$, 
\item $\varepsilon_i(b) + \varepsilon_i^*(b) + \langle\alpha_i^\vee,\wt(b)\rangle \ge 2$ implies $f_if_i^*(b) = f_i^*f_i(b)$. 
\end{enumerate} 
\end{Proposition}

\begin{Corollary}[{\cite[Cor.~1.5]{TW12}}]
For any fixed $i\in I$ and $b\in B(\infty)$, the subset of $B(\infty)$ that can be reached from $b$ by applying sequences of the operators $e_i,f_i,e_i^*,f_i^*$ is of the following form, where the width of the diagram at the bottom is $\langle \alpha_i^\vee, \wt(b_{\mathrm{top}}) \rangle$ and $b_{\mathrm{top}}$ is the vertex at the top of the diagram. Here the width is $4$.
\[
\begin{tikzpicture}[xscale=1.0,yscale=.65]
\foreach \x in {1,...,5}
{\foreach \y in {1,...,\x}
 {\node[dot] (\y\x) at (-\x+2*\y,-\x) {};}
  \node[dot] (\x6) at (2*\x-5,-6.25) {}; 
}
\path[->,thick]
 (11) edge (22) edge (33) edge (44) edge (55)
 (12) edge (23) edge (34) edge (45)
 (13) edge (24) edge (35)
 (14) edge (25);
% (-3,-2) edge node[above]{$f_i$} (-2,-2);
\path[->,dotted,thick]
 (11) edge (12) edge (13) edge (14) edge (15)
 (22) edge (23) edge (24) edge (25)
 (33) edge (34) edge (35)
 (44) edge (45);
% (-3,-3) edge node[above]{$f_i^*$} (-2,-3);
%\path[->,dashed,thick]
% (4,-3) edge node[above]{$f_i=f_i^*$} (5,-3);
\foreach \x in {1,...,5}
{\path[->,dashed,thick] (\x5) edge (\x6);
 \path[->,dashed,thick] (\x6) edge (2*\x-5,-7.25);}
\end{tikzpicture}
\]
Here $f_i$ acts on a vertex by following the solid or dashed arrow, and $f_i^*$ acts by following the dotted or dashed arrow. 
\end{Corollary}

\subsection{Reduced expressions and convex orders}

\begin{Definition}
A total order $\prec$ on $\Phi^+$ is called \defn{convex} if, for all triples of roots $\beta, \beta',\beta''$ with $\beta'=\beta+\beta''$, we have either $\beta\prec\beta'\prec\beta''$ or $\beta''\prec\beta'\prec\beta$.
\end{Definition}

\begin{Theorem}[\cite{papi}]\label{thm:papi}
There is a bijection between $R(w_0)$ and convex orders on $\Phi^+$: if $w_0 = s_{i_1}s_{i_2}\cdots s_{i_N}$, then the corresponding convex order $\prec$
is 
\[
\beta_1 = \alpha_{i_1} \ \ \prec \ \ 
\beta_2 = s_{i_1}\alpha_{i_2} \ \ \prec \ \ \cdots \ \ \prec \ \ 
\beta_N = s_{i_1}s_{i_2}\cdots s_{i_{N-1}}\alpha_{i_N}.
\]
\end{Theorem}

By Theorem \ref{thm:papi}, we identity convex orderings with reduced expressions of $w_0$.

\begin{Lemma} \label{lem:combine-orders} Fix two convex orders $\prec, \prec'$ on $\Phi^+$ such that, for some root $\beta$, 
$$\{ \alpha  \in \Phi^+: \alpha \prec \beta \} = \{ \alpha  \in \Phi^+: \alpha \prec' \beta \}.$$
Call this set $X$. One can make a hybrid convex order by ordering $X$ according to $\prec$ and $\Phi^+ \backslash X$ according to $\prec'$. 
\end{Lemma}

\begin{proof}
Let $|X|=k$, and consider the reduced words $\ii = (i_1,\dots,i_k,i_{k+1},\dots,i_N)$ and $\ii' = (i'_1,\dots,i'_k,i_{k+1}',\dots,i_N')$ related to $\prec$ and $\prec'$ respectively as in Theorem \ref{thm:papi}. Then $\ii'' = (i_1,\dots,i_k,i_{k+1}',\dots,i_N')$ is reduced and corresponds to the required convex order. 
\end{proof}

\begin{Definition} \label{def:lex-orders}
Fix an order $i_1, \ldots, i_n$ on $I$. Define a corresponding order on $\Phi^+$ as follows. 
For $\beta, \beta' \in \Phi^+$, let 
\[
\beta = \sum_{i\in I} p_i \alpha_i \ \ \ \text{ and } \ \ \ 
\beta'= \sum_{i\in I} p'_i \alpha_i.
\]
Then $\beta \prec \beta'$ if
\begin{enumerate}

\item $\min\{ k: p_{i_k} \neq 0 \} < \min\{ k: p'_{i_k} \neq 0 \} $, or

\item $\min\{ k: p_{i_k} \neq 0 \} = \min\{ k: p'_{i_k} \neq 0 \} $, and, for that $k$, 
$$\displaystyle \left(\frac{p_{i_{k+1}}}{p_{i_k}}, \frac{p_{i_{k+2}}}{p_{i_k}} , \dots, \frac{p_{i_n}}{p_{i_k}}\right) < \left(\frac{p'_{i_{k+1}}}{p'_{i_k}}, \frac{p'_{i_{k+2}}}{p'_{i_k}} , \dots, \frac{p'_{i_n}}{p'_{i_k}}\right)$$ in lexicographical order. 
\end{enumerate}
\end{Definition}

\begin{Example}
Consider type $D_4$ and the enumeration $i_1=1$, $i_2=2$, $i_3=3$, and $i_4=4$. The corresponding order on $\Phi^+$ from Definition \ref{def:lex-orders} is
\[
1 \prec 12 \prec 124 \prec 123 \prec 1234 \prec 12234 \prec 2 \prec 24 \prec 23 \prec 234 \prec 3 \prec 4,
\]
where, for example, $124$ means $\alpha_1+\alpha_2+\alpha_4$. 
This is different from the convex order corresponding to this enumeration in, for example, \cite{Leclerc} (which orders roots $c_1 \alpha_1+c_2\alpha_2+c_3\alpha_3+c_4\alpha_4$ by lexicographically ordering the set $(\frac{c_1}{c},\frac{c_2}{c},\frac{c_3}{c},\frac{c_4}{c})$ where $c=c_1+c_2+c_3+c_4$) since, in particular, the roots $124$ and $123$ are reversed. 
\end{Example}

\begin{Lemma} \label{lem:lexcon}
For any enumeration of $I$, the order on $\Phi^+$
from Definition \ref{def:lex-orders} is convex.
\end{Lemma}

\begin{proof}
It is immediate that $\prec$ defines a total order on $\Phi^+$. Fix $\beta = \sum_{i \in I} p_i\alpha_i$, $\beta' = \sum_{i \in I} p_i'\alpha_i$, and $\beta'' = \sum_{i \in I} p_i''\alpha_i$ such that $\beta' = \beta+\beta''$. Without loss of generality, assume $\beta \prec \beta''$. Let $m =\text{min} \{ k: p_{i_k} \neq 0 \} $, and similarly for $m'$ and $m''$. 

If $m < m''$, then $m'= m< m''$, so $\beta'\prec\beta''$. Also, $p_{i_m}=p'_{i_m}$ and, for each $j>m$, $p_
{i_j}' \geq p_{i_j}$, with at least one of these inequalities being strict, so $\beta\prec\beta'$.

If $m=m''$, then this is also equal to $m'$. Let $s>m$ be minimal such that 
$\displaystyle 
\frac{p_{i_s}}{p_{i_m}} < \frac{p''_{i_s}}{p''_{i_m}}.
$
\[ \text{Clearly } \qquad  
\frac{p_{i_s}}{p_{i_m}} < \frac{p'_{i_s}}{p'_{i_m}} <  \frac{p''_{i_s}}{p''_{i_m}}
\ \ \ \text{ and } \ \ \
\frac{p_{i_j}}{p_{i_m}} = \frac{p'_{i_j}}{p'_{i_m}} = \frac{p''_{i_j}}{p''_{i_m}}
\]
for all $m<j<s$,
so $\beta\prec\beta'\prec\beta''$, as required. 
\end{proof}

\subsection{PBW bases and enumerations of $B(\infty)$ by Kostant partitions} \label{sec:PBWe}
For $c\in \bz_{> 0}$, define 
\[
F_i^{(c)} := \frac{F_i^c}{[c]!} \quad \text{ where} \quad [c]! := \prod_{j=1}^c
\frac{q^{j}-q^{-j}}{q-q^{-1}}.
\]
Given $\ii = (i_1,\dots,i_N)\in R(w_0)$ and $\cc = (c_\beta^\ii \in \bz_{\ge0}^N : \beta\in\Phi^+)$, define
\begin{equation}\label{eq:lustparam}
F_\ii^{\cc} = F_{\ii:\beta_1}^{(c_{\beta_1}^\ii)}F_{\ii:\beta_2}^{(c_{\beta_2}^\ii)}\cdots F_{\ii:\beta_N}^{(c_{\beta_N}^\ii)}
\ \ \ 
\text{ where }
\ \ \ 
F_{\ii:\beta_k}^{(c_{\beta_k}^\ii)} = T_{i_1}T_{i_2}\cdots T_{i_{k-1}}(F_{i_k}^{(c_{\beta_k}^\ii)}),
\end{equation}
and $T_i$ is the Lusztig automorphism of $U_q(\mathfrak{g})$ defined in \cite[Section 37.1.3]{L10} (there, it is denoted $T''_{i,-1}$). Then the set $\CB_\ii = \{ F_\ii ^{\cc} : \cc \in \bz_{\ge0}^N \}$ forms a $\bq(q)$-basis of $U_q^-(\mathfrak{g})$, called the \defn{PBW basis}.  
The notation $\beta_k$ used in the subscript of $F_{\ii:\beta_k}$ is because, for all $k$,  
\begin{equation} 
\wt(F_{{\bf i}, \beta_k}) = -s_{i_1} \cdots s_{i_{k-1}}\alpha_{i_k} = -\beta_k. 
\end{equation}
These are exactly the negative roots. We index the root vectors by the corresponding positive roots $\beta_k$. When the context is clear we omit the subscript $\ii$ in $F_{\ii:\beta_k}$. 

\begin{Theorem}[\cite{saito}]
For $\ii \in R(w_0)$, 
$\mathrm{Span}_{\mathcal{A}} (\CB_\ii)  = L(\infty)$ and
$\CB_\ii+ q L(\infty)=B(\infty)$. 
\end{Theorem}

\begin{Definition} \label{def:bbb}
For $b \in B(\infty)$, the \defn{Lusztig data} associated to $b$ is the tuple $\cc^\ii(b) \in \bz_{\ge0}^N$ such that $F^{\cc^\ii(b)}_\ii+ q L(\infty)=b$.
\end{Definition}

This gives a parametrization of $B(\infty)$ by $\ii$-Lusztig data for any $\ii$. 

\begin{Proposition}[\cite{BZ01,L10}]
\label{prop:crystal_op_PBW}
Fix $b \in B(\infty)$ and $i \in I$.
\begin{enumerate}
\item Let $\ii \in R(w_0)$ be such that, in the corresponding convex order, $\beta_1 = \alpha_i$.  If $\cc^\ii(b) = (c_{\beta_1},c_{\beta_2},\dots,c_{\beta_N})$, then $\cc^\ii(f_i b) = (c_{\beta_1}+1,c_{\beta_2},\dots,c_{\beta_N})$.  If $c_{\beta_1} = 0$ then $e_i b =\ghost$ and otherwise $\cc^\ii(e_i b) = (c_{\beta_1}-1,c_{\beta_2},\dots,c_{\beta_N})$.
\item Let $\ii \in R(w_0)$ be such that, in the corresponding convex order, $\beta_N = \alpha_i$.  If $\cc^\ii(b) = (c_1,c_2,\dots,c_N)$ then $\cc^\ii(f_i^* b) = (c_{\beta_1},c_{\beta_2},\dots,c_{\beta_N}+1)$.  If $c_{\beta_N} = 0$ then $e_i^* b =\ghost$ and otherwise $\cc^\ii(e_i^* b) = (c_{\beta_1},c_{\beta_2},\dots,c_{\beta_N}-1)$.
\end{enumerate}
\end{Proposition}

In order to use Proposition \ref{prop:crystal_op_PBW} to understand the whole crystal structure on $B(\infty)$, we need to understand how the Lusztig data $\cc^\ii(b)$ are related for different $\ii$. Since all reduced expressions are related by a sequence of braid moves, it is enough to understand what happens to $\cc^\ii(b)$ when $\ii$ changes by a single braid move.

\begin{Lemma}
\label{Lem:braid_on_roots}
Fix $\ii \in R(w_0)$ and let $\{\beta_1 \prec \cdots \prec \beta_N\}$ be the corresponding convex ordering of $\Phi^+$.
\begin{enumerate}
\item There is a reduced expressions $\ii'$ related to $\ii$ by a $2$-term braid move $(i_k,i_{k+1}) \to (i_{k+1},i_k)$ if and only if $(\beta_k|\beta_{k+1})=0$.    In this case, $\beta_k \prec \beta_{k+1}$ is replaced by $\beta_{k+1} \prec' \beta_k$ after the braid move.
\item \label{ht2} There is a reduced expressions $\ii'$ related to $\ii$ by a braid move $(i_k,i_{k+1},i_{k+2}) \to (i_{k+1},i_k,i_{k+1})$, with $i_k = i_{k+2}$, if and only if $\{\beta_k,\beta_{k+1},\beta_{k+2}\}$ form a root system of type $\mathfrak{sl}_3$.  In this case, $\beta_k \prec \beta_{k+1} \prec \beta_{k+2}$ is replaced by $\beta_{k+2} \prec' \beta_{k+1} \prec' \beta_k$ after the braid move.
\item \label{ht3} There is a reduced expressions $\ii'$ related to $\ii$ by a braid move $(i_k,i_{k+1},i_{k+2},i_{k+3}) \to (i_{k+1},i_k,i_{k+1},i_k)$, with $i_k = i_{k+2}$ and $i_{k+1} = i_{k+3}$, if and only if $\{\beta_k,\beta_{k+1},\beta_{k+2},\beta_{k+3}\}$ form a root system of type $B_2$.  In this case, $\beta_k \prec \beta_{k+1} \prec \beta_{k+2} \prec \beta_{k+3}$ is replaced by $\beta_{k+3} \prec' \beta_{k+2} \prec' \beta_{k+1} \prec' \beta_k$ after the braid move.
\end{enumerate} 
\end{Lemma}

\begin{proof}
Let $w = s_{i_1}s_{i_2}\cdots s_{i_{k-1}}$ and suppose $(\beta_k|\beta_{k+1}) = 0$.  Then
\[
0 = (\beta_k|\beta_{k+1}) = (w\alpha_{i_k} | ws_{i_k}\alpha_{i_{k+1}}) = (\alpha_{i_k} | s_{i_k}\alpha_{i_{k+1}}) = (s_{i_k}\alpha_{i_k} | \alpha_{i_{k+1}}) = - (\alpha_{i_k} | \alpha_{i_{k+1}}).
\]
Hence $a_{i_k,i_{k+1}} =0$, and we can perform a $2$-term braid move exactly when $a_{i_k,i_{k+1}} =0$. Now $\beta_{k+1} = w\alpha_{i_{k+1}}$ and $\beta_k = ws_{i_{k+1}}\alpha_{i_k}$, so reversing $i_k$ and $i_{k+1}$ clearly reverses these. 

The proof of \eqref{ht2} is similar: both the condition that $\{\beta_k,\beta_{k+1},\beta_{k+2}\}$ forms a root system of type $\mathfrak{sl}_3$ and the ability to perform a $3$-term braid move of the desired form are equivalent to $a_{i_k,i_{k+1}} = -1$.  In this situation, with $w$ as above,
\begin{align*}
\beta_{k+2} &= ws_{i_k}s_{i_{k+1}}\alpha_{i_k} = ws_{i_k}(\alpha_{i_k}+\alpha_{i_{k+1}}) = w\alpha_{i_{k+1}}, \\
\beta_{k+1} &= ws_{i_k}\alpha_{i_{k+1}} = w(\alpha_{i_k} + \alpha_{i_{k+1}}) = ws_{i_{k+1}}\alpha_{i_k}, \\
\beta_k &= w\alpha_{i_k} = ws_{i_{k+1}}(\alpha_{i_k} + \alpha_{i_{k+1}}) = ws_{i_{k+1}}s_{i_k}\alpha_{i_{k+1}}.
\end{align*}
By Theorem \ref{thm:papi}, we have $\beta_{k+2} \prec' \beta_{k+1} \prec' \beta_k$.

The proof of \eqref{ht3} is also similar.
\end{proof}

\begin{Lemma}[{\cite[\S 2.1]{L90}}] 
\label{Lem:moves} 
Let $b\in B(\infty)$.
\begin{enumerate}
\item If $\mathfrak{g}=\mathfrak{sl}_2 \times \mathfrak{sl}_2$, then the Lusztig data $\cc(b)$ and $\cc'(b)$ with respect to the two possible reduced expressions are identical (meaning $c_{\alpha_1}=c'_{\alpha_1}$ and $c_{\alpha_2}=c'_{\alpha_2}$, but the order is reversed). 

\item If $\mathfrak{g}=\mathfrak{sl}_3$, then the Lusztig data $\cc(b)$ and $\cc'(b)$ with respect to the two possible reduced expressions  are related by
\begin{align*}
c'_{\alpha_1} &= \max \{ c_{\alpha_1+\alpha_2}, c_{\alpha_1}+c_{\alpha_1+\alpha_2}-c_{\alpha_2} \},\\
c'_{\alpha_1+\alpha_2} &= \min \{ c_{\alpha_1}, c_{\alpha_2} \},\\
c'_{\alpha_2} &= \max \{c_{\alpha_1+\alpha_2}, c_{\alpha_2}+c_{\alpha_1+\alpha_2}-c_{\alpha_1} \}.
\end{align*}
\end{enumerate}
More generally, if a reduced expression is changed by a $2$- or $3$-term braid move, the Lusztig data changes according to these rules for the affected roots. 
\end{Lemma}

\begin{Lemma}[{\cite[Thm.~3.1]{BZ97}}] \label{lem:LBZ}
\label{Lem:movesBC} 
If $\mathfrak{g}= B_2$, with $\alpha_1$ being the long simple root, and $\cc(b)$ is the Lusztig data of $b\in B(\infty)$ with respect to one reduced expression, then the Lusztig data $\cc'(b)$ with respect to the other reduced expression is 
\begin{align*}
c'_{\alpha_1} &= c_{\alpha_1} + c_{\alpha_1+\alpha_2} + c_{\alpha_1+2\alpha_2} -\pi_1,\\
c'_{\alpha_1+\alpha_2} &= 2\pi_1 - \pi_2,\\
c'_{\alpha_1+2\alpha_2} &= \pi_2 - \pi_1,\\
c'_{\alpha_2} &= c_{\alpha_1+\alpha_2} + 2 c_{\alpha_1+2\alpha_2} + c_{\alpha_2} - \pi_2,
\end{align*}
where
\begin{align*}
\pi_1 &= \min\{ c_{\alpha_1} + c_{\alpha_1+\alpha_2}, c_{\alpha_1} + c_{\alpha_2}, c_{\alpha_1+2\alpha_2} + c_{\alpha_2} \}, \\
\pi_2 &= \min\{ 2 c_{\alpha_1} + c_{\alpha_1+\alpha_2}, 2 c_{\alpha_1} + c_{\alpha_2}, 2 c_{\alpha_1+2\alpha_2} + c_{\alpha_2} \}.
\end{align*}
More generally, if a reduced expression is changed by a $2$- or $3$-term braid move, the Lusztig data changes according to these rules for the affected roots. 
\end{Lemma}

The following is equivalent to Lemma \ref{lem:LBZ}, but is more suited to our purposes. 

\begin{Lemma} \label{lem:newlfs}
Let $\mathfrak{g}=B_2$, with $\alpha_1$ the long root and $\alpha_2$ the short root. If $\cc(b) = ( c_{\alpha_1}, c_{\alpha_1+\alpha_2}, c_{\alpha_1+2\alpha_2}, c_{\alpha_2} )$ is the Lusztig data for $b \in B(\infty)$ with respect to one reduced expression, then the Lusztig data $\cc'(b)$ with respect to the other reduced expression can be found as follows: Consider the string of large and small brackets
\[
T(\cc)= \underbrace{)\cdots)}_{c_{\alpha_1+\alpha_2}}\ \
\underbrace{\Big (\cdots\Big (}_{c_{\alpha_1}}\ \ 
\underbrace{\Big )\cdots \Big )}_{c_{\alpha_1+2 \alpha_{2}}}\ \
\underbrace{(\cdots(}_{c_{\alpha_1+\alpha_2}}\ \
\underbrace{)\cdots)}_{c_{\alpha_2}}\ \
\underbrace{\Big (\cdots \Big (}_{c_{\alpha_1+2\alpha_2}}.\ \
\]
Cancel brackets using the rules:
\begin{itemize}
\item cancel as many pairs $()$ and $\Big ( \Big )$ as possible, and

\item cancel remaining ``$\Big ($" and ``$)$" using the rule that one ``$\Big ($" can cancel either one or two ``$)$", and as many total brackets as possible are canceled.
\end{itemize}
Then 
\begin{equation} \label{leq-bra}
\begin{aligned}
c'_{\alpha_1} &=  \#\,\text{uncanceled } (\  +\  \#\,\text{uncanceled } \Big (  \\
c'_{\alpha_1+\alpha_2} &= \#\,\text{canceled }(  ) \ +\  \# \,\text{canceled } \Big ( ) \\
c'_{\alpha_1+2\alpha_2} &= \#\,\text{canceled }\Big (\Big)\ +\  \#\,\text{canceled } \Big ( )) \\
 c'_{\alpha_2} &= \#\,\text{uncanceled } )\ +\ 2 \#\,\text{uncanceled } \Big ).
\end{aligned}
\end{equation}
Equivalently, 
\begin{equation} \label{leq-alg}
\begin{aligned}
c'_{\alpha_1}  &= \max \{ c_{\alpha_1+2\alpha_2}, c_{\alpha_1+\alpha_2} + c_{\alpha_1+2\alpha_2}- c_{\alpha_2}, 
 c_{\alpha_1}+ c_{\alpha_1+\alpha_2} - c_{\alpha_2} 
 \}  \\
c'_{\alpha_1+\alpha_2} &= \min \{\max \{ c_{\alpha_1+\alpha_2}, c_{\alpha_1}+ c_{\alpha_1+\alpha_2} - c_{\alpha_1+2\alpha_2} \},  c_{\alpha_2} \} \\
 c'_{\alpha_1+2\alpha_2} &= \min \{\max \left\{ c_{\alpha_1+2\alpha_2},  c_{\alpha_2}+ 2 c_{\alpha_1+2\alpha_2} - c_{\alpha_1+\alpha_2} - c_{\alpha_1} \right\},  c_{\alpha_1} \} \\
  c'_{\alpha_2} &= \max \{c_{\alpha_1+\alpha_2}, 2 c_{\alpha_1+2\alpha_2}+ c_{\alpha_1+\alpha_2}-2 c_{\alpha_1}, c_{\alpha_2}+ 2 c_{\alpha_1+2\alpha_2}  
-2 c_{\alpha_1}  \} .\\
\end{aligned}
\end{equation}
More generally, if a reduced expression is changed by a $4$-term braid move, then the Lusztig data changes according to these rules for the affected roots. 
\end{Lemma}

\begin{proof}
That \eqref{leq-bra} and \eqref{leq-alg} agree is straightforward. We must show that they agree with Lemma \ref{lem:LBZ}. 
The first and fourth formulas in \eqref{leq-alg} agree with the corresponding formulas in Lemma \ref{lem:LBZ} by a straightforward calculation. Since $\alpha_1+\alpha_2$ and $\alpha_1+2\alpha_2$ are linearly independent, it suffices to show \eqref{leq-bra} gives data of the correct weight. So, let $\wt(b)= x \alpha_1+ y \alpha_2$. Clearly
\begin{equation}
 x= \# (\ +\  \# \Big ( \quad \text{ and } \quad y= \# )\ +\  2 \# \Big ).
\end{equation}
But then 
\begin{equation}
\begin{aligned}
c'_{\alpha_2} + 2 c'_{\alpha_1+2\alpha_2} + c'_{\alpha_1+\alpha_2} &= \# )\ +\  2 \# \Big ) = y, \text{ and}  \\
c'_{\alpha_1} + c'_{\alpha_1+\alpha_2} + c'_{\alpha_1+2\alpha_2} &= \# (\ +\  \# \Big (= x.
\end{aligned}
\end{equation}
So, $\wt(\cc')$ is in fact correct. 
\end{proof}

\begin{Remark}\label{rem:newlfs}
The analogue of Lemma \ref{lem:newlfs} in the case where $\mathfrak{g} = C_2$, $\alpha_1$ is the short root, and $\alpha_2$ is the long root can be obtained by interchanging the roles of $\alpha_1$ and $\alpha_2$ and reversing the equalities in Equation \eqref{leq-bra}; that is, if $\cc(b) = (c_{\alpha_1}, c_{2\alpha_1+\alpha_2}, c_{\alpha_1+\alpha_2}, c_{\alpha_2})$ is the Lusztig data for $b \in B(\infty)$ with respect to one reduced expression, then the Lusztig data $\cc'(b)$ with respect to the other reduced expression can be found using bracketing sequence
\[
T(\cc)= \underbrace{)\cdots)}_{c_{\alpha_1+\alpha_2}}\ \
\underbrace{\Big (\cdots\Big (}_{c_{\alpha_2}}\ \ 
\underbrace{\Big )\cdots \Big )}_{c_{2\alpha_1+\alpha_{2}}}\ \
\underbrace{(\cdots(}_{c_{\alpha_1+\alpha_2}}\ \
\underbrace{)\cdots)}_{c_{\alpha_1}}\ \
\underbrace{\Big (\cdots \Big (}_{c_{2\alpha_1+\alpha_2}}
\]
and setting
\begin{equation} \label{leq-bra-C}
\begin{aligned}
 c'_{\alpha_1} &= \#\,\text{uncanceled } )\ +\ 2 \#\,\text{uncanceled } \Big ) \\
 c'_{2\alpha_1+\alpha_2} &= \#\,\text{canceled }\Big (\Big)\ +\  \#\,\text{canceled } \Big ( )) \\
c'_{\alpha_1+\alpha_2} &= \#\,\text{canceled }(  ) \ +\  \# \,\text{canceled } \Big ( ) \\
 c'_{\alpha_2} &=  \#\,\text{uncanceled } (\  +\  \#\,\text{uncanceled } \Big (  .
\end{aligned}
\end{equation}
\end{Remark}

\section{Examples: Calculating crystal operators on PBW bases in general} \label{sec:exg}

For any two reduced expressions, we can understand the map $R_{\ii}^{\ii'}\colon \bz_{\ge0}^N \longrightarrow \bz_{\ge0}^N$ sending $\cc^\ii(b)$ to $\cc^{\ii'}(b)$ by finding a way to move from $\ii$ to $\ii'$ by a sequence of braid moves, and composing the maps from \S\ref{sec:PBWe}. This gives a way to calculate any $f_i$.

\begin{Example}\label{ex:non_adapted} 
Let $\mathfrak{g}$ be of type $A_4$. Then $\ii = (1,3,2,1,3,2,4,3,2,1)$ is a reduced expression, and the corresponding order on $\Phi^+$ is
\[
1 \prec 3 \prec 123 \prec 23 \prec 12 \prec 2 \prec 1234 \prec 234 \prec 34 \prec 4,
\]
where $1$ is identified with $\alpha_1$, $1234$ with $\alpha_1+\alpha_2+\alpha_3+\alpha_4$, and so on.  Consider
\[
b = F_1^{(1)} F_3^{(1)} F_{123}^{(5)} F_{23}^{(3)} F_{12}^{(2)} F_{2}^{(3)} F_{1234}^{(4)} F_{234}^{(0)} F_{34}^{(1)} F_{4}^{(1)} \in \CB_\ii.
\]
Calculating $f_1 b$ is easy: the exponent of $F_1$ just increases by $1$.
We now compute $f_2b$:
\[
\arraycolsep=5pt
\begin{array}{rccccccccccl}
b &=& F_1^{(1)} & F_3^{(1)} & F_{123}^{(5)} & F_{23}^{(3)} & F_{12}^{(2)} & F_{2}^{(3)} & F_{1234}^{(4)} & F_{234}^{(0)} & F_{34}^{(1)} & F_{4}^{(1)} \\[5pt]
&\simeq& {\color{red}F_3^{(1)}} & {\color{red}F_1^{(1)}} & F_{123}^{(5)} & F_{23}^{(3)} & F_{12}^{(2)} & F_{2}^{(3)} & F_{1234}^{(4)} & F_{234}^{(0)} & F_{34}^{(1)} & F_{4}^{(1)} \\[5pt]
&\simeq& F_3^{(1)} & {\color{red}F_{23}^{(7)}} & {\color{red}F_{123}^{(1)}} & {\color{red}F_{1}^{(5)}} & F_{12}^{(2)} & F_{2}^{(3)} & F_{1234}^{(4)} & F_{234}^{(0)} & F_{34}^{(1)} & F_{4}^{(1)} \\[5pt]
&\simeq& F_3^{(1)} & F_{23}^{(7)} & F_{123}^{(1)} & {\color{red}F_{2}^{(2)}} & {\color{red}F_{12}^{(3)}} & {\color{red}F_{1}^{(4)}} & F_{1234}^{(4)} & F_{234}^{(0)} & F_{34}^{(1)} & F_{4}^{(1)} \\[5pt]
&\simeq& F_3^{(1)} & F_{23}^{(7)} & {\color{red}F_{2}^{(2)}} & {\color{red}F_{123}^{(1)}} & F_{12}^{(3)} & F_{1}^{(4)} & F_{1234}^{(4)} & F_{234}^{(0)} & F_{34}^{(1)} & F_{4}^{(1)} \\[5pt]
&\simeq& {\color{red}F_{2}^{(8)}} & {\color{red}F_{23}^{(1)}} & {\color{red}F_{3}^{(7)}} & F_{123}^{(1)} & F_{12}^{(3)} & F_{1}^{(4)} & F_{1234}^{(4)} & F_{234}^{(0)} & F_{34}^{(1)} & F_{4}^{(1)}, \\[5pt]
f_2b &\simeq& {\color{blue}F_{2}^{(9)}} & F_{23}^{(1)} & F_{3}^{(7)} & F_{123}^{(1)} & F_{12}^{(3)} & F_{1}^{(4)} & F_{1234}^{(4)} & F_{234}^{(0)} & F_{34}^{(1)} & F_{4}^{(1)} \\[5pt]
&&&&&& \vdots \\[5pt]
&\simeq& F_1^{(1)} & F_3^{(1)} & {\color{blue}F_{123}^{(4)}} & {\color{blue}F_{23}^{(4)}} & {\color{blue}F_{12}^{(3)}} & F_{2}^{(3)} & F_{1234}^{(4)} & F_{234}^{(0)} & F_{34}^{(1)} & F_{4}^{(1)}.
\end{array}
\]
The first five steps perform braid moves and modify the PBW monomial according to the piecewise linear functions from Lemma \ref{Lem:moves}.  (To see this, recall that, by Lemma \ref{Lem:braid_on_roots}, two roots can be interchanged by a $2$-term braid move exactly if they are perpendicular, and a $3$-term braid move applies to consecutive roots $\beta, \beta',\beta''$ if and only if $\beta'=\beta+\beta''$.) Then $F_2$ is on the left, so to get $f_2b$ just add one to its exponent. Then perform braid moves and the corresponding pieces linear functions to get back to the original order.
\end{Example}

\begin{Example}\label{ex:pbw_compute}
Let $\mathfrak{g}$ be of type $D_4$, and $\ii = (1,2,3,4,2,1,2,3,4,2,3,4)$, where $2$ is the trivalent node (see Figure \eqref{fig:diagrams}).  The corresponding order on $\Phi^+$ is
\[
1 \prec 12 \prec 123 \prec 124 \prec 1234 \prec 12234 \prec 2 \prec 24 \prec 23 \prec 234 \prec 3 \prec 4,
\]
where $1$ is identified with $\alpha_1$, $12$ with $\alpha_1 + \alpha_2$, and so on. Consider 
\[
b= F_1^{(2)}  F_{12}^{(1)}  F_{123}^{(4)} F_{124}^{(2)} F_{1234}^{(1)} F_{12234}^{(3)} F_{2}^{(3)}  F_{24}^{(1)} F_{23}^{(2)} F_{234}^{(1)} F_{3}^{(2)} F_{4}^{(0)} \in \CB_\ii,
\]
The calculation of $f_4 b$, goes as follows:
\[
\arraycolsep=3pt
\begin{array}{rccccccccccccl}
b &= & F_1^{(2)} & F_{12}^{(1)} & F_{123}^{(4)} & F_{124}^{(2)} & F_{1234}^{(1)} & F_{12234}^{(3)} & F_{2}^{(3)} & F_{24}^{(1)} & F_{23}^{(2)} & F_{234}^{(1)} & F_{3}^{(2)} & F_{4}^{(0)} \\[5pt]
&\simeq & F_1^{(2)} & F_{12}^{(1)} & {\color{red}F_{124}^{(2)}} & {\color{red}F_{123}^{(4)}} & F_{1234}^{(1)} & F_{12234}^{(3)} & F_{2}^{(3)} & F_{24}^{(1)} & F_{23}^{(2)} & F_{234}^{(1)} & {\color{red}F_{4}^{(0)}} & {\color{red}F_{3}^{(2)}} \\[5pt]
&\simeq & F_1^{(2)} & F_{12}^{(1)} & F_{124}^{(2)} & F_{123}^{(4)} & F_{1234}^{(1)} & F_{12234}^{(3)} & F_{2}^{(3)} & F_{24}^{(1)} & {\color{red}F_{4}^{(1)}} & {\color{red}F_{234}^{(0)}} & {\color{red}F_{23}^{(3)}} & F_{3}^{(2)} \\[5pt]
&&&&&&& \vdots \\[5pt]
 &\simeq & {\color{red}F_4^{(2)}} & {\color{red} F_{1}^{(2)}} & F_{124}^{(1)} & F_{12}^{(2)} & F_{1234}^{(1)} & F_{123}^{(4)} & F_{12234}^{(3)} & F_{24}^{(1)} & F_{2}^{(3)} & F_{234}^{(0)} & F_{23}^{(3)} & F_{3}^{(2)}, \\[5pt]
f_4 b &\simeq & {\color{blue}F_4^{(3)}} & F_{1}^{(2)} & F_{124}^{(1)} & F_{12}^{(2)} & F_{1234}^{(1)} & F_{123}^{(4)} & F_{12234}^{(3)} & F_{24}^{(1)} & F_{2}^{(3)} & F_{234}^{(0)} & F_{23}^{(3)} & F_{3}^{(2)} \\[5pt]
&&&&&&& \vdots \\[5pt]
&\simeq & F_1^{(2)} & F_{12}^{(1)} & {\color{blue}F_{123}^{(3)}} & F_{124}^{(2)} & {\color{blue}F_{1234}^{(2)}} & F_{12234}^{(3)} & F_{2}^{(3)} & F_{24}^{(1)} & F_{23}^{(2)} & F_{234}^{(1)} & F_{3}^{(2)} & F_{4}^{(0)}.
\end{array}
\]
\end{Example}

\begin{Example}\label{ex:C_PBW}
In type $C_3$ where $\alpha_3$ is the long root (see Figure \eqref{fig:diagrams}), set $\ii = (1,2,3,2,1,2,3,2,3)$.  The corresponding order on $\Phi^+$ is
\[
1 \prec 12 \prec 11223 \prec 123 \prec 1223 \prec 2 \prec 223 \prec 23 \prec 3,
\]
where $1$ is identified with $\alpha_1$, $11223$ is identified with $2\alpha_1+2\alpha_2+\alpha_3$, and so on.  Consider
\[
b = F_1^{(4)} F_{12}^{(1)} F_{11223}^{(3)} F_{123}^{(2)} F_{1223}^{(0)} F_{2}^{(0)} F_{223}^{(5)} F_{23}^{(1)} F_{3}^{(2)} \in \CB_\ii.
\]
Computing $f_3$:
\[
\arraycolsep=5pt
\begin{array}{rcccccccccl}
b &=& F_1^{(4)} & F_{12}^{(1)} & F_{11223}^{(3)} & F_{123}^{(2)} & F_{1223}^{(0)} & F_{2}^{(0)} & F_{223}^{(5)} & F_{23}^{(1)} & F_{3}^{(2)} \\[5pt]
&\simeq& F_1^{(4)} & F_{12}^{(1)} & F_{11223}^{(3)} & F_{123}^{(2)} & F_{1223}^{(0)} & {\color{red}F_{3}^{(6)}} & {\color{red}F_{23}^{(0)}} & {\color{red}F_{223}^{(2)}} & {\color{red}F_{2}^{(7)}} \\[5pt]
&\simeq& F_1^{(4)} & F_{12}^{(1)} & F_{11223}^{(3)} & F_{123}^{(2)} & {\color{red}F_{3}^{(6)}} & {\color{red}F_{1223}^{(0)}} & F_{23}^{(0)} & F_{223}^{(2)} & F_{2}^{(7)} \\[5pt]
&\simeq& F_1^{(4)} & {\color{red}F_{3}^{(7)}} & {\color{red}F_{123}^{(1)}} & {\color{red}F_{11223}^{(3)}} & {\color{red}F_{12}^{(2)}} & F_{1223}^{(0)} & F_{23}^{(0)} & F_{223}^{(2)} & F_{2}^{(7)} \\[5pt]
&\simeq& {\color{red}F_3^{(7)}} & {\color{red}F_{1}^{(4)}} & F_{123}^{(1)} & F_{11223}^{(3)} & F_{12}^{(2)} & F_{1223}^{(0)} & F_{23}^{(0)} & F_{223}^{(2)} & F_{2}^{(7)}, \\[5pt]
f_3b &\simeq& {\color{blue}F_3^{(8)}} & F_{1}^{(4)} & F_{123}^{(1)} & F_{11223}^{(3)} & F_{12}^{(2)} & F_{1223}^{(0)} & F_{23}^{(0)} & F_{223}^{(2)} & F_{2}^{(7)} \\[5pt]
&\simeq& {\color{red}F_1^{(4)}} & {\color{blue}F_{3}^{(8)}} & F_{123}^{(1)} & F_{11223}^{(3)} & F_{12}^{(2)} & F_{1223}^{(0)} & F_{23}^{(0)} & F_{223}^{(2)} & F_{2}^{(7)} \\[5pt]
&\simeq& F_1^{(4)} & {\color{red}F_{12}^{(1)}} & {\color{red}F_{11223}^{(3)}} & {\color{red}F_{123}^{(2)}} & {\color{blue}F_{3}^{(7)}} & F_{1223}^{(0)} & F_{23}^{(0)} & F_{223}^{(2)} & F_{2}^{(7)} \\[5pt]
&\simeq& F_1^{(4)} & F_{12}^{(1)} & F_{11223}^{(3)} & F_{123}^{(2)} & {\color{red}F_{1223}^{(0)}} & {\color{blue}F_{3}^{(7)}} & F_{23}^{(0)} & F_{223}^{(2)} & F_{2}^{(7)} \\[5pt]
&\simeq& F_1^{(4)} & F_{12}^{(1)} & F_{11223}^{(3)} & F_{123}^{(2)} & F_{1223}^{(0)} & {\color{red}F_{2}^{(0)}} & {\color{blue}F_{223}^{(4)}} & {\color{blue}F_{23}^{(3)}} & {\color{red}F_{3}^{(2)}}.
\end{array}
\]
Notice that the exponent of $F_{23}$ increases by two. 
\end{Example}

\section{Results}\label{sec:main}

We now have a method to calculate crystal operators on any PBW monomial, but it can be a long process. However, for some words, some $f_i$ can be calculated by a simple bracketing procedure. We discuss those words here. In fact, in most types there are a few very nice words where all the $f_i$ are easy to calculate.

\subsection{Simply braided words and bracketing rules}
Fix $\ii \in R(w_0)$. The elements of $B(\infty)$ are parameterized by their $\ii$-Lusztig data, so we can think of $B(\infty)$ as a crystal structure on $\bz^N$. Specifically,  $f_i(\cc) = \cc'$ if and only if there is some $b\in B(\infty)$ such that $\cc = \cc^\ii(b)$ and $\cc' = \cc^\ii(f_ib)$.  

\begin{Definition} \label{def:sa} 
Fix a reduced expression $\ii$ for $w_0$, and $i \in I$. We say that $\ii$ is \defn{simply braided for $i$} if one can perform a sequence of braid moves to $\ii$ to get to a word $\ii'$ with $i'_1=i$, and each move in the sequence is either
\begin{itemize}
\item a $2$-term braid move, or
\item a braid move such that $\alpha_i$ is the rightmost of the roots affected.
\end{itemize}
We call $\ii$ \defn{simply braided}  if it is simply braided for all $i \in I$. 
\end{Definition}

\begin{Remark}
Examples \ref{ex:pbw_compute} and \ref{ex:C_PBW} both show computations in simply braided cases. 
\end{Remark}

Fix $i\in I$ and $\ii \in R(w_0)$ which is simply braided for $i$. Fix a sequence of braid moves as in Definition \ref{def:sa}, and let $M_1, \ldots, M_k$ be the non-trivial braid moves. For each $\ell$, let $\Delta_\ell$ be the corresponding rank $2$ root system.

\begin{Definition}
\label{def:bracket_rule}
Assume $\ii \in R(w_0)$ is simply braided for $i$. 
Fix an $\ii$-Lusztig datum $\cc$. For each $\ell$, let $\cc^\ell$ be the rank $2$ Lusztig datum defined by $c^\ell_{\alpha_i}=0$, and for all other $\beta \in \Delta_\ell$, $c^\ell_\beta=c_\beta$. Let
\[
R_{i;\ell}(\cc) = \varepsilon_i(\cc^\ell)
\ \ \ \ \ \text{ and } \ \ \ \ \ 
L_{i;\ell}(\cc) = \langle \alpha_i^\vee, \wt(\cc) \rangle + \varepsilon_i(\cc^\ell)+  \varepsilon^*_i(\cc^\ell).
\]
Set 
\[
S_i(\cc)= 
\underbrace{)\cdots)}_{R_{i;k}} \ 
\underbrace{(\cdots(}_{L_{i;k}} \
\cdots\
\underbrace{)\cdots)}_{R_{i;2}} \
\underbrace{(\cdots(}_{L_{i;2}} \
\underbrace{)\cdots)}_{R_{i;1}} \
\underbrace{(\cdots(}_{L_{i;1}} \
\underbrace{) \cdots )}_{c_{\alpha_i}}.
\]
\end{Definition}

\begin{Remark}
The integer $L_{i;\ell}(\cc)$ is the number of times one may apply $f_i$ to $\cc^\ell$ before $c^\ell_{\alpha_i}$ becomes nonzero, and is sometimes called ``jump,'' see \cite{LV11}. 
\end{Remark}

\begin{Theorem} \label{th:brgeneral}
Fix $i\in I$ and assume $\ii \in R(w_0)$ simply braided for $i$. Then, for any $\ii$-Lusztig datum $\cc$,
$f_i(\cc)$ is 
\begin{itemize}
\item the $\ii$-Lusztig datum where $\cc^k$ is changed to $f_i(\cc^k)$ if the leftmost uncanceled ``$($" in $S_i(\cc)$ comes from a root in $\Delta_k$, or
\item the $\ii$-Lusztig datum where $c_{\alpha_i}$ has increased by $1$ if there is no uncanceled ``$($". 
\end{itemize}
\end{Theorem}

\begin{proof}
Proceed by induction on the number $m_i(\ii)$ of braid moves required to change $\ii$ to an $i$-adapted word, using only moves of the form allowed by Definition \ref{def:sa}.  The case where $i_1=i$ is trivial. So, assume $\ii'$ is related to $\ii$ by an allowable braid move and that $m_i(\ii')= m_i(\ii)-1$. If the braid move from $\ii$ to $\ii'$ is a $2$-term braid move, then it does not change the order of the brackets in $S_i(\cc)$ and, by Lemmas \ref{Lem:moves} and \ref{lem:newlfs} and Remark \ref{rem:newlfs}, it does not change any $c_\beta,$ so does not change the values of $L_{i;j}$ and $R_{i;j}$.  This is true both before and after applying $f_i$, so the result follows by induction.

If the braid move involves more then two terms, then by definition it must be the move $M_1$. First consider the case when the leftmost uncanceled ``$($" is in $\Delta_j$ for $j>1$. Applying the move $M_1$ gives a new Lusztig datum, $\cc'$, where $c'_{\alpha_i}$ is the number of uncanceled right brackets in the substring $)^{R_{i;1}(\cc)} (^{L_{i;1}(\cc)} )^{c_{\alpha_i}}$. Doing this doesn't change the placement of the first uncanceled ``$($". By induction $f_i$ affects a root not in $\Delta_1$, and acts as in the statement. Undoing the move $M_1$ gives the result. 

It remains to consider the cases when either the leftmost uncanceled ``$($" is in $\Delta_1$ or there is no uncanceled ``$($". Again start by applying $M_1$. The resulting Lusztig datum has no uncanceled ``$($", so by induction applying $f_i$ just adds one to $c_{\alpha_i}$.  Undoing $M_1$ gives the result of applying $f_i$. But this procedure has exactly applied $f_i$ to the Lusztig data ${\cc^1}$. 
See Figure \ref{fig:induction}.
\end{proof}

\begin{figure}
\[
\begin{array}{cccccccccl}
F_1^{(1)} & F_{12}^{(2)} & F_{123}^{(1)} & F_{1234}^{(5)} & F_{2}^{(2)} & F_{23}^{(3)}  & F_{3}^{(1)} & F_{234}^{(4)}& F_{34}^{(1)} & F_{4}^{(1)} \\
 & ) & ((& & ))) & ((&   )&&    & \\[12pt]

F_1^{(1)} & F_{12}^{(2)} & F_{123}^{(1)} & F_{1234}^{(5)} & {\color{red}F_{3}^{(3)}} & {\color{red}F_{23}^{(1)}} & {\color{red}F_{2}^{(4)}} & F_{234}^{(4)}  & F_{34}^{(1)} & F_{4}^{(1)} \\
 & ) & ((& & ))) & &   &&    & 
 \\[12pt]

F_1^{(1)} & F_{12}^{(2)} & F_{123}^{(1)} & F_{1234}^{(5)} & {\color{blue}F_{3}^{(4)}} & F_{23}^{(1)} & F_{2}^{(4)} & F_{234}^{(4)}  & F_{34}^{(1)} & F_{4}^{(1)} \\ 
& ) & ((& & ))){\color{blue}\boldsymbol{)}} & &   &&    & \\[12pt]

F_1^{(1)} & F_{12}^{(2)} & F_{123}^{(1)} & F_{1234}^{(5)} & {\color{blue}F_{2}^{(1)}} & {\color{blue}F_{23}^{(4)}} & {\color{red}F_{3}^{(1)}} & F_{234}^{(4)}  & F_{34}^{(1)} & F_{4}^{(1)} \\ & ) & ((& & )))) & (& )  &&    & \\[12pt]
\end{array}
\]
\caption{\label{fig:induction}The inductive step in the proof of Theorem \ref{th:brgeneral}.  Here we are applying $f_3$ in type $A_3$. At each step we have placed the string of brackets $S_3( \cc )$ below the monomial $F^\cc$ (note, for example, that the brackets for the roots $F_{12}$ and $F_{123}$ need to be reversed). This is the final case, where the first uncanceled bracket is in $\cc^1$. The first braid move is the move $M_1$, and involves the roots $2$, $23$, and $3$. The new string of brackets after that move is obtained by deleting the brackets corresponding to these three roots, and replacing them with the number of uncanceled right brackets in the string we would use to apply $f_3$ in that rank 2 case. There were no uncanceled left brackets further to the left, so now there are no uncanceled left brackets at all, and, by induction, $f_3$ increases the exponent of $F_3$ by $1$. Then we undo the move $M_1$. The result is the result of applying $f_3$ to the rank  2 monomial $F_2^{(2)} F_{23}^{(3)} F_3^{(1)}$.   }
\end{figure}

\begin{Remark}
Theorem \ref{th:brgeneral} implies that the sub-root systems $\Delta_1, \dots, \Delta_k$ don't depend on the sequence of braid moves (provided it satisfies the conditions of Definition \ref{def:sa}).
\end{Remark}

\begin{Remark}
The reduced expression in Example \ref{ex:non_adapted} is not simply braided, and hence the action of $f_2$ is more complicated: it affects three roots, and it follows from Theorem \ref{th:brgeneral} that this cannot happen in simply braided simply-laced cases.
\end{Remark}

\subsection{Existence of simply braided words}

In this section, we show that any ``good enumeration" of a Dynkin diagram, as defined in \cite{Lit98}, gives a simply braided word. These words are closely related to the Littelmann's ``nice decompositions." In particular, such words exist in all types except $E_8$ and $F_4$. 

\begin{Definition} \label{def:ti}
For $J\subseteq I$, let $W_J$ denote the parabolic subgroup of $W$ generated by $J$.  Fix $i \in I$. Let $\tau^i$ be the minimal length representative of $w_0$ in $W_{I \backslash \{ i \}}\backslash W$. 
\end{Definition}

For each subset $I'$ of $I$, let $\sigma^{I'}$ be the diagram automorphism of $I'$ defined by $-w_0^{I'}$, where $w_0^{I'}$ is the longest element of the Weyl group for $I'$. 
Note that $\tau^i= (w_0^{I \backslash \{ i \}})^{-1}w_0$, so $(\tau^i)^{-1}$ takes simple roots other than $\alpha_i$ to simple roots according to the map $(\sigma^I)^{-1} \sigma^{I \backslash \{ i \}}$. 

\begin{Definition} \label{def:t(j)}
For any enumeration $i_1,\dots,i_n$ of $I$, let 
$\tau^{(k)} = (\sigma^I)^{-1} \sigma^{I \backslash \{ i_1, \ldots, i_{k-1} \} }w$, where $w$ is the minimal length representative of $w_0^{I \backslash \{ i_1, \ldots, i_{k-1} \}}$ in $W_{I\setminus\{i_1,\dots,i_k\}}\backslash W_{I\setminus\{i_1,\dots,i_{k-1}\}}$.
\end{Definition}

\begin{Lemma} \label{rem:ct}
The reduced expression of $w_0$ corresponding to the convex order from Definition \ref{def:lex-orders} factors as $\tau^{(1)} \tau^{(2)} \cdots \tau^{(n)}$.
\end{Lemma}

\begin{Example} \label{ex:yet-another} 
Consider type $D_5$, where $3$ is the trivalent node, and the enumeration $i_1=4$, $i_2=2$, $i_3=3$, $i_4=1$, and $i_5=5$. The corresponding order on roots is
\begin{multline*}
4 \prec  34 \prec 345 \prec  234 \prec 2345 \prec 1234 \prec 12345 \prec 23345 \prec 123345 \prec 1223345 \\
\prec 2 \prec 12 \prec 23 \prec 235 \prec 123 \prec 1235 \prec 3 \prec 35 \prec 1 \prec 5,
\end{multline*}
which corresponds to the reduced expression
\[
s_4 s_3 s_5 s_2 s_3 s_1s_2 s_4 s_3 s_5 s_3 s_4 s_2 s_1 s_3 s_2 s_4 s_3 s_1 s_4.
\]
Then 
\begin{gather*}
s_4 s_3 s_5 s_2 s_3 s_1s_2 s_4 s_3 s_5 = \tau^{(1)}, \\
s_3 s_4 s_2 s_1 s_3 s_2 =  \sigma^{-1} \sigma^{\{ 1,2,3,5\}}  s_2 s_1 s_3 s_5 s_2 s_3 = \tau^{(2)}, \\
s_4s_3 = \sigma^{-1} \sigma^{\{ 1,3,5\}}  s_3s_5= \tau^{(3)}, \ s_1=\tau^{(4)}, \ \text{ and } \ s_4=\tau^{(5)}.
\end{gather*} 
Here $\sigma$ is the diagram automorphism swapping $4$ and $5$, $\sigma^{\{1,2,3,5\}}$ reverses the order of that $A_4$ Dynkin diagram, $\sigma^{\{1,3,5\}}$ swaps $3$ and $5$, and $\sigma^{\{1,5\}}$ and $\sigma^{\{5\}}$ are trivial. 
Notice that all pairs of roots of the form $\beta, \beta+\alpha_5$ are adjacent, as in Lemma \ref{lem:faci1} below. 
\end{Example}

\begin{proof}[Proof of Lemma \ref{rem:ct}]
The convex order from Definition \ref{def:lex-orders} is of the form
\[
\beta_{1,1} \prec \beta_{1,2} \prec \cdots \prec \beta_{1,k_1} \prec \beta_{2,1} \prec \cdots \prec \beta_{1,k_2} \prec \cdots \prec \beta_{(n-1), k_{n-1}} \prec \beta_{n,1},
\]
where 
\begin{itemize}
\item $\{ \beta_{1,1} , \beta_{1,2} ,\dots,\beta_{1,k_1}  \}$ consists of exactly those roots $\beta = \sum_{j\in I} c_j \alpha_j$ with $c_{i_1} >0$, 

\item $\{ \beta_{2,1}, \dots ,\beta_{2,k_2}  \}$ consists of roots $\beta= \sum_{j\in I} c_j \alpha_j$ with $c_{i_1}=0$ and $c_{i_2} >0$, and so on.
\end{itemize}
Consider the corresponding reduced expression, factored as
$w^{(1)} w^{(2)} \cdots w^{(n)}$, where $w^{(j)}$ corresponds to the roots $\beta_{j,1}, \ldots, \beta_{j,k_j}$. 
Then
$\{ \beta_{1,1} , \beta_{1,2} , \dots, \beta_{1,k_1}  \}$ is the subset of $\Phi^+$ sent to negative roots by $(w^{(1)})^{-1}$, so in particular $(w^{(1)})^{-1}$ does not send any simple root other than $\alpha_{i_1}$ to a negative root. This implies that every reduced expression of $(w^{(1)})^{-1}$ has $s_{i_1}$ on the right, so $(w^{(1)})^{-1}$ is a minimal length coset representative in $W/W_{I \backslash \{ i_1 \}}$. Equivalently, $w^{(1)}$ is a minimal length coset representative in $W_{I \backslash \{ i_1 \}}\backslash W$.

For any $w \in W$ longer than $w^{(1)}$, $w^{-1}$ must take some root not in $\{ \beta_{1,1} , \beta_{1,2} ,\dots,\beta_{1,k_1}  \}$ to a negative root, so, since the coefficient of $\alpha_{i_1}$ in that root is $0$, it must take some other simple root to a negative root. This implies that $w$ is not a minimal length representative of its coset in $W_{I \backslash \{ i_1 \}}\backslash W$. Thus $w^{(1)}$ is in fact equal to $\tau^{(1)}$. 

Let $a\colon I \backslash \{ i_1\} \longrightarrow I$ denote the map $(\sigma^I)^{-1}\sigma^{I \backslash \{ i_1\}}$. Since $\tau^{(1)}$ performs the map $a^{-1}$ on roots, $w^{(2)} \cdots w^{(n)}$ must be the reduced expression for the root system $a (I \backslash \{ i_1\})$ corresponding the the convex order
\[
a \beta_{2,1} \prec \cdots \prec a \beta_{1,k_2} \prec \cdots \prec a  \beta_{(n-1), k_{n-1}} \prec  a\beta_{n,1}.
\]
By induction, this factors as $\eta^{(2)} \cdots \eta^{(n)}$ where 
\[
\eta^{(k)}=  a  (\sigma^{I \backslash \{ i_1 \}})^{-1}  \sigma^{I \backslash \{  i_1, \dots,  i_{k-1} \}} w
\]
and $w$ is the minimal length representative of $w_0^{I \backslash \{ i_1, \ldots, i_{k-1} \}}$ in $W_{I\setminus\{i_1,\dots,i_k\}}\backslash W_{I\setminus\{i_1,\dots,i_{k-1}\}}$.
The result follows by substituting $a =(\sigma^I)^{-1}\sigma^{I \backslash \{ i_1\}}$ and simplifying. 
\end{proof}

\begin{Lemma} \label{lem:faci1}
For any enumeration $i_1, \ldots ,i_n$ of $I$, the convex order from Definition \ref{def:lex-orders} and Lemma \ref{lem:lexcon} has the property that, for any rank two sub root system $\Delta_\circ$ with simple roots $\beta$ and $ \alpha_{i_n}$, all the roots except possibly $\alpha_{i_n}$ are adjacent. 
\end{Lemma}

\begin{proof} 
This is immediate from Definition \ref{def:lex-orders}, since all roots in span $\beta, \alpha_n$ other than $\beta$ are equal in all the comparisons but the very last, so they are adjacent.  
\end{proof}

\begin{Definition}
Following Littelmann \cite{Lit98}, $i \in I$ is called \defn{braidless}, if for any $\tau \in W$, the following holds: if $\alpha,\gamma \in \Phi$ are such that $\langle \alpha^\vee, \tau\omega_i \rangle > 0$ and $\langle \gamma^\vee, \tau\omega_i \rangle > 0$, then $\langle \alpha^\vee, \gamma \rangle = 0$. Here $\omega_i$ is the corresponding fundamental weight. 
\end{Definition}

\begin{Definition}
An enumeration $i_1, \ldots, i_n$ of $I$ is called a \defn{good enumeration} if each $i_k$ is braidless for the root system generated by  $\{\alpha_{i_k},\dots,\alpha_{i_n}\}$. 
\end{Definition}

By \cite[Lemma 3.1]{Lit98}, a node of a Dynkin diagram is braidless exactly if the corresponding fundamental weight is either minuscule or co-minuscule (meaning it is minuscule for the Dynkin diagram with arrows reversed), or the Dynkin diagram is rank $2$. Using the classification of minuscule weights in Table \ref{tab:minuscule} we see:

\begin{Proposition}  \label{prop-good-list}
The following is a complete list of good enumerations of $I = \{i_1,\dots,i_n\}$ depending on the underlying type.
\begin{enumerate}
\item If $\mathfrak{g}$ is of type $A_n$, then any order of $I$ is a good enumeration.
\item If $\mathfrak{g}$ if of type $B_n$ or $C_n$, then order $I$ as $1 < 2 < 3 \cdots < k < n < \cdots$ for some $k \geq 0$, where everything after the $n$ is arbitrary. The $k=0$ case means $n$ is first in the order. 
\item If $\mathfrak{g}$ is of type $D_n$, then order $I$ as $1<2<\cdots<k<x<\cdots$ for some $k \geq 0$, where $x=n-1$ or $n$, and everything after $x$ is arbitrary.
\item If $\mathfrak{g}$ is of type $E_6$, then set $i_1$ to be one of the two vertices farthest from the trivalent vertex in the Dynkin diagram, and then follow type $D_5$ rules.
\item If $\mathfrak{g}$ is of type $E_7$, then set $i_1$ to be the vertex farthest from the trivalent node, and then follow type $E_6$ rules. \qed
\end{enumerate}
\end{Proposition}

\begin{table}[t]
\[
\begin{array}{cl}\toprule
A_n & 1,\dots,n \\
B_n & n \\
C_n & 1 \\
D_n & 1,n-1,n \\
E_6 & 1,6 \\
E_7 & 7 \\
E_8,F_4,G_2 & \text{none} \\ \bottomrule
\end{array}
\]
\label{tab:minuscule}
\caption{The indices of minuscule fundamental weights, using the indexing of fundamental weights from \cite{bourbaki}.  This table is from \cite[p.~132]{bourbaki7-9}.}
\end{table}

Braidless nodes are useful for use because of the following property.
\begin{Lemma}[{\cite[Lemma 3.2]{Lit98}}] 
If $i$ is braidless for $I$, then $\tau^i$ has a unique reduced expression, up to $2$-term braid moves. \qed
\end{Lemma}

\begin{Lemma}\label{lemma:reverse_Littelmann}
For any good enumeration of $I$, the corresponding convex order from Lemma \ref{def:lex-orders} is simply braided.
\end{Lemma}

\begin{proof} 
Fix $i$, and let $k$ be such that $i_k=i$. Recall from Lemma \ref{rem:ct} that the corresponding reduced expression factors as $\tau^{(1)} \tau^{(2)} \cdots \tau^{(n)}$. Let $X$ be the set of roots corresponding to $\tau^{(1)} \cdots \tau^{(k-1)}$. Consider two convex orders defined using Lemma \ref{lem:lexcon}:
\begin{itemize}
\item $\prec$ defined using the order $i_1, i_2 \ldots, i_n $, and
\item $\prec'$ defined by  $i_1, \dots, i_{k-1}, i_{k+1} , \dots, i_n,  i_k$.
\end{itemize}
Consider the
hybrid order $\prec''$ where $X$ is ordered according to $\prec'$, and the rest of the roots are ordered according to $\prec$. By Lemma \ref{lem:combine-orders} this remains convex. Then:
\begin{enumerate}

\item $\prec$ and $\prec''$ both factor as a product of $\tau^{(j)}$ in the same order, so, since we are working with a good enumeration, differ by $2$-term braid moves;

\item the first root coming from $\tau^{(k)}$ in $\prec$ is $\alpha_i$;

\item  by Lemma \ref{lem:faci1}, all roots $\beta$ that are $\prec'' \alpha_i$ which satisfy $( \beta | \alpha_i ) <0$ have the property that the roots in the rank $2$ root system including $\beta$ and $\alpha_i$ are adjacent according to $\prec''$. From this it follows that $\prec''$ is simply braided for $i$. 
\end{enumerate}
Together these certainly imply that $\prec$ is simply braided for $i$.
See Example \ref{ex:3en}.
\end{proof}

\begin{Example} \label{ex:3en}
Consider type $D_4$ and the good enumeration $i_1=1$, $i_2=3$, $i_3=2$, and $i_4=4$. We illustrate the proof that the convex order from Lemma \ref{def:lex-orders} is simply braided for node $3$. The convex orders considered are
\begin{gather*}
1\prec12\prec124\prec123\prec1234\prec12234\prec3\prec23\prec234\prec2\prec24\prec4 \\
1\prec'12\prec'123\prec'124\prec'1234\prec'12234\prec'2\prec'23\prec'24\prec'234\prec'4\prec'3,
\end{gather*}
and the hybrid order
\[
1\prec''12\prec''123\prec''124\prec''1234\prec''12234
\prec''3\prec''23\prec''234\prec''2\prec''24\prec''4
\]
which follows $\prec'$ up to $12234$, then follows $\prec$. The order $\prec'$ is simply braided for $3$ by Lemma \ref{lem:faci1}, so $\prec''$ is simply braided for $3$ as well because it is the same order as $\prec'$ to the left of the root $3$.   
Both $\prec$ and $\prec''$ factor as a product of $\tau^{(j)}$ in same order, so, since we are using a good enumeration, they differ only by $2$-term braid moves, and hence $\prec$ is simply braided for 3 by definition. This last step fails for other enumerations. 
\end{Example}

\subsection{Bracketing rules in rank 2} \label{ss:brrank2}
We will use the notation $f_i(\cc)_\beta$ to mean the $\beta$-component of $f_i(\cc)$, where $\cc$ is an $\ii$-Lusztig datum for some $\ii \in R(w_0)$ and $\beta$ is a positive root.

\begin{Lemma}
\label{Lem:bracket2}
Let $\mathfrak{g}=\mathfrak{sl}_3$ and set $\ii = (1,2,1)$.  For any $\ii$-Lusztig datum $\cc = (c_{\alpha_1},c_{\alpha_1+\alpha_2},c_{\alpha_2})$, $f_2(\cc)$ can be calculated as follows. Make the string of brackets
\[
S_2(\cc)= \quad
\underbrace{)\cdots)}_{c_{\alpha_1+\alpha_2}}\ \
\underbrace{(\cdots(}_{c_{\alpha_{1}}}\ \
\underbrace{)\cdots)}_{c_{\alpha_{2}}}.
\]
Then $\cc$ and $f_2(\cc)$ differ as follows.
\begin{itemize}
\item If there is an uncanceled ``$($", then $f_2(\cc)_{\alpha_1+\alpha_2}= c_{\alpha_1+\alpha_2}+1$ and
$f_2(\cc)_{\alpha_1}= c_{\alpha_1}-1$.
 \item If there is no uncanceled ``$($", then $f_2(\cc)_{\alpha_2}= c_{\alpha_2}+1$.
\end{itemize}
\end{Lemma}

\begin{proof}
Let $\cc' = (c_{\alpha_2}', c_{\alpha_1+2\alpha_2}', c_{\alpha_1}') = R_{121}^{212}(\cc)$, where $R_\ii^{\ii'}$ is the transition map from Lemma \ref{Lem:moves}. If there is an uncanceled left bracket, then $c_{\alpha_1} > c_{\alpha_2}$.  Hence 
$
\cc' = (c_{\alpha_1+\alpha_2} ,c_{\alpha_2},c_{\alpha_1}+c_{\alpha_1+\alpha_2}-c_{\alpha_2}).
$  
Using Proposition \ref{prop:crystal_op_PBW},
\[
f_2(\cc') = (c_{\alpha_1+\alpha_2}+1 ,c_{\alpha_2},c_{\alpha_1}+c_{\alpha_1+\alpha_2}-c_{\alpha_2}).
\]
Applying the inverse to $R_{121}^{212}$ gives
$
f_2(\cc) = (c_{\alpha_1}-1, c_{\alpha_1+\alpha_2}+1, c_{\alpha_2}),
$
as required.  

If there is no uncanceled left bracket, then $c_{\alpha_1} \le c_{\alpha_2}$, so
$
\cc' = (c_{\alpha_1+\alpha_2} + c_{\alpha_2} - c_{\alpha_1}, c_{\alpha_1} , c_{\alpha_1+\alpha_2} ).
$
Using Proposition \ref{prop:crystal_op_PBW}, 
\[
f_2(\cc') = (c_{\alpha_1+\alpha_2} + c_{\alpha_2} - c_{\alpha_1} + 1, c_{\alpha_1} , c_{\alpha_1+\alpha_2} ).
\]
Applying the inverse to $R_{121}^{212}$ gives
$
f_2(\cc) = (c_{\alpha_1}, c_{\alpha_1+\alpha_2}, c_{\alpha_2}+1),
$
as required.
\end{proof}

\begin{Lemma} \label{lem:bmb1}
Let $\mathfrak{g}$ be of type $B_2$ where $\alpha_2$ is the short root and set $\ii = (1,2,1,2)$.  For any $\ii$-Lusztig datum $\cc = (c_{\alpha_1},c_{\alpha_1+\alpha_2},c_{\alpha_1+2\alpha_2},c_{\alpha_2})$, $f_2(\cc)$ can be calculated as follows.  Make the string of brackets
\[
S_2(\cc)= \quad
\underbrace{)\cdots)}_{c_{\alpha_1+\alpha_2}}\ \
\underbrace{(\cdots(}_{2c_{\alpha_{1}}}\ \
\underbrace{)\cdots)}_{2c_{\alpha_{1}+2\alpha_2}}\ \
\underbrace{(\cdots(}_{c_{\alpha_{1}+\alpha_2}}\ \
\underbrace{)\cdots)}_{c_{\alpha_2}}.
\]
Then $\cc$ and $f_2(\cc)$ differ as follows.
\begin{enumerate}
\item \label{eac1} If the leftmost uncanceled ``$($" corresponds to $\alpha_1$, then 
$f_2(\cc)_{\alpha_1} = c_{\alpha_1}-1$ and 
$f_2(\cc)_{\alpha_1+\alpha_2} = c_{\alpha_1+\alpha_2}+1$.
\item \label{eac2} If the leftmost uncanceled ``$($" corresponds to $\alpha_1+\alpha_2$, then $f_2(\cc)_{\alpha_1+\alpha_2} = c_{\alpha_1+\alpha_2}-1$ and
$f_2(\cc)_{\alpha_1+2\alpha_2} = c_{\alpha_1+2\alpha_2}+1$.
\item \label{eac3} If there is no uncanceled ``$($", then $f_2(\cc)_{\alpha_2} = c_{\alpha_2}+1$.
\end{enumerate}
\end{Lemma}

\begin{proof}
Let $\overline\cc$ denote the Lusztig data proposed for $f_2(\cc)$. Let $\cc'$ be the Lusztig data obtained by performing a braid move on $\cc$, and $\overline \cc'$ the data obtained by performing a braid move on $\overline \cc$. It suffices to show that $\cc'$ and $\overline \cc'$ differ only by $\overline c'_{\alpha_2}= c'_{\alpha_2}+1$. We compare the strings of brackets for $T(\cc)$ and $T(\bar \cc)$ from Lemma \ref{lem:newlfs}. The cases are broken down according to the itemization of the statement.  

In case (\ref{eac1}), 
\[
T(\overline \cc)= 
\underbrace{)\cdots) {\color{blue} )}}_{c_{\alpha_1+\alpha_2}+1}\ \
%{\color{red} \Big (}
\underbrace{\Big (\cdots\Big ( }_{c_{\alpha_1}-1}\ \
\underbrace{\Big )\cdots \Big )}_{c_{\alpha_1+2 \alpha_2}}\ \
\underbrace{{\color{blue} (} (\cdots(}_{c_{\alpha_1+\alpha_2}+1}\ \
\underbrace{)\cdots)}_{c_{\alpha_2}}\ \
\underbrace{{\Big (\cdots \Big (}}_{c_{\alpha_1+2 \alpha_2}}.
\]
Here one bracket corresponding to $\alpha_1$ has been removed and one of each type of bracket corresponding to $\alpha_1+\alpha_2$ have been added. 
The condition that the leftmost uncanceled ``$($'' in $S_2(\cc)$ corresponds to $\alpha_1$ implies one of the following.
\begin{itemize}
\item There is a pair $\Big ()$ in $T(\cc)$.  Then $T(\overline \cc)$ has one less $\Big ( )$, one more $()$, and one more $)$. 

\item The leftmost uncanceled left bracket in $T(\cc)$ comes from $\alpha_1$ and there is no $\Big ()$. Then $T(\overline \cc)$ has one less uncanceled $\Big ($ and one more uncanceled $($ and $)$. The new $($ is uncanceled since otherwise there would have been a $\Big ( )$. 
\end{itemize}

In case (\ref{eac2}), 
\[
T(\overline \cc)=
\underbrace{)\cdots) %{\color{red} )}
}_{c_{\alpha_1+\alpha_2 }-1}\ \
\underbrace{\Big (\cdots\Big ( }_{c_{\alpha_1}} \ \
\underbrace{\Big )\cdots \Big ){\color{blue} \Big )}}_{c_{\alpha_1+2 \alpha_2}+1}\ \
\underbrace{ %{\color{red} (}
(\cdots(}_{c_{\alpha_1+\alpha_2}-1}\ \
\underbrace{)\cdots)}_{c_{\alpha_2}}\ \
\underbrace{{\color{blue} \Big (} \Big (\cdots \Big (}_{c_{\alpha_1+2 \alpha_2}+1}.
\]
Comparing to $T(\cc)$, there is one less uncanceled (, one less uncanceled ), one more uncanceled $\Big )$, and one more uncanceled $\Big ($.

In case (\ref{eac3}), 
\[
T(\overline \cc)=
\underbrace{)\cdots) }_{c_{\alpha_1+\alpha_2}}\ \
\underbrace{ \Big (\cdots\Big ( }_{c_{\alpha_1}}\ \
\underbrace{\Big )\cdots \Big )}_{c_{\alpha_1+2 \alpha_2}}\ \
\underbrace{ (\cdots(}_{c_{\alpha_1+\alpha_2}}\ \
\underbrace{)\cdots) {\color{blue} )}}_{c_{\alpha_2}+1}\ \
\underbrace{ \Big (\cdots \Big (}_{c_{\alpha_1+2 \alpha_2}}.
\]
There is one new uncanceled $)$.

In each case, it is easy to see using Lemma \ref{lem:newlfs} that $\overline \cc'$ is as desired. 
\end{proof}

\begin{Lemma} \label{lem:bmb2}
Let $\mathfrak{g}$ be of type $C_2$ where $\alpha_2$ is the long root and set $\ii = (1,2,1,2)$.  For any $\ii$-Lusztig datum $\cc = (c_{\alpha_1},c_{2\alpha_1+\alpha_2},c_{\alpha_1+\alpha_2},c_{\alpha_2})$, $f_2(\cc)$ can be calculated as follows.  Make the string of brackets
\[
S_2(\cc)= \quad
\underbrace{)\cdots)}_{c_{2\alpha_1+\alpha_2}}\ \
\underbrace{(\cdots(}_{c_{\alpha_{1}}}\ \
\underbrace{)\cdots)}_{c_{\alpha_{1}+\alpha_2}}\ \
\underbrace{(\cdots(}_{c_{2\alpha_{1}+\alpha_2}}\ \
\underbrace{)\cdots)}_{c_{\alpha_2}}.
\]
Then $\cc$ and $f_2(\cc)$ differ as follows.
\begin{enumerate}
\item \label{hgf1} If the leftmost uncanceled ``$($" corresponds to $\alpha_1$ and $c_{\alpha_1+\alpha_2}= c_{\alpha_1}-1$, then $f_2(\cc)_{\alpha_1}= c_{\alpha_1} - 1 \text{ and }
f_2(\cc)_{\alpha_1+\alpha_2} = c_{\alpha_1+\alpha_2}+1.$ 

\item \label{hgf2} If the leftmost uncanceled ``$($" corresponds to $\alpha_1$ and $c_{\alpha_1+\alpha_2}< c_{\alpha_1}-1$, 
 then 
$f_2(\cc)_{\alpha_1} =c_{\alpha_1}-2 \text{ and } f_2(\cc)_{2\alpha_1+\alpha_2} = c_{2\alpha_1+\alpha_2}+1.$
\item \label{hgf3} If the leftmost uncanceled ``$($" corresponds to $2\alpha_1+\alpha_2$, then $f_2(\cc)_{2\alpha_1+\alpha_2} = c_{\alpha_1+\alpha_2}-1$ and
$f_2(\cc)_{\alpha_1+\alpha_2} = c_{\alpha_1+2\alpha_2}+2$.
\item \label{hgf4} If there is no uncanceled ``$($", then $f_2(\cc)_{\alpha_2} = c_{\alpha_2}+1$.
\end{enumerate}
\end{Lemma}

\begin{proof} 
Let $\overline \cc$ denote the Lusztig data proposed for $f_2(\cc)$, and $\cc'$, $\overline \cc'$ the data obtained by performing a braid move on $\cc$ and $\overline \cc$ respectively. It suffices to show that $\cc'$ and $\overline \cc'$ differ only by $\overline c'_{\alpha_2}= c'_{\alpha_2}+1$. 
Consider the string of brackets for $T(\cc)$ from Lemma \ref{lem:newlfs}, where, since here $\alpha_1$ is the long root, the roles of the two roots are interchanged. The roots in $T(\cc)$ come in the reverse order of those in $S_2(\cc)$, so the conditions on the left brackets in $S_2(\cc)$ give information about right brackets in $T(\cc)$.

In case (\ref{hgf1}),
\[
 T(\overline \cc)=
\underbrace{)\cdots){\color{blue} )}}_{c_{\alpha_2+\alpha_1}+1}\ \
\underbrace{\Big (\cdots\Big ( }_{c_{\alpha_2}}\ \
\underbrace{\Big )\cdots \Big )}_{c_{\alpha_2+2 \alpha_1}}\ \
\underbrace{{\color{blue} (} (\cdots(}_{c_{\alpha_2+\alpha_1}+1}\ \
\underbrace{)\cdots)%{\color{red} )}
}_{c_{\alpha_1}-1}\ \
\underbrace{\Big (\cdots \Big (}_{c_{\alpha_2+2 \alpha_1}}.
\]
The conditions imply that all of the brackets added and removed are uncanceled, so in total $T(\bar \cc)$ has one more $($ than $T(\cc)$.

In case (\ref{hgf2}),
\[
T(\overline \cc)=
\underbrace{)\cdots)}_{c_{\alpha_2+\alpha_1}}\ \
\underbrace{\Big (\cdots\Big ( }_{c_{\alpha_2}}\ \
\underbrace{\Big )\cdots \Big ){\color{blue} \Big )}}_{c_{\alpha_2+2 \alpha_1}+1}\ \
\underbrace{(\cdots(}_{c_{\alpha_2+\alpha_1}}\ \
\underbrace{)\cdots)%{\color{red} ))}
}_{c_{\alpha_1}-2}\ \
\underbrace{ {\color{blue} \Big (} \Big(\cdots \Big (}_{c_{\alpha_2+2 \alpha_1}+1}.
\]
This breaks into two cases:
\begin{itemize}

\item $c_{\alpha_2+2\alpha_1} \geq c_{\alpha_2}$, so there are no uncanceled $\Big)$ nor $\Big( ))$. Then there are two less $)$, one more $\Big )$, and one more $\Big ($. 

\item $c_{\alpha_2+2\alpha_1} < c_{\alpha_2}$. Since there is an uncanceled right bracket corresponding to $\alpha_1$, $T(\cc)$ has a $\Big( ))$. So there are one more $\Big ( \Big )$ and $\Big ($ and one less $\Big( ))$. 

\end{itemize}

In case (\ref{hgf3}),
\[
 T(\overline \cc)=
\underbrace{)\cdots){\color{blue} ))}}_{c_{\alpha_2+\alpha_1}+2}\ \
\underbrace{\Big (\cdots\Big ( }_{c_{\alpha_2}}\ \
\underbrace{\Big )\cdots \Big ) %{\color{red} \Big )}
}_{c_{\alpha_2+2 \alpha_1}-1}\ \
\underbrace{{\color{blue} ((} (\cdots(}_{c_{\alpha_2+\alpha_1}+2}\ \
\underbrace{)\cdots)}_{c_{\alpha_1}}\ \
\underbrace{ %{\color{red} \Big (} 
\Big (\cdots \Big (}_{c_{\alpha_2+2 \alpha_1}-1}.
\]
There are two more uncanceled ) and (, and one less uncanceled $\Big )$ and ${\Big (}$.  

For case (\ref{hgf4}),
\[
 T(\overline \cc)=
\underbrace{)\cdots)}_{c_{\alpha_2+\alpha_1}}\ \
\underbrace{{\color{blue} \Big (} \Big (\cdots\Big ( }_{c_{\alpha_2}+1}\ \
\underbrace{\Big )\cdots \Big )}_{c_{\alpha_2+2 \alpha_1}}\ \
\underbrace{ (\cdots(}_{c_{\alpha_2+\alpha_1}}\ \
\underbrace{)\cdots)}_{c_{\alpha_1}}\ \
\underbrace{\Big (\cdots \Big (}_{c_{\alpha_2+2 \alpha_1}}.
\]
The string $T(\cc)$ has no uncanceled $)$ and no $\Big ( ))$, so the new $\Big( $ is uncanceled.

In each case, it is easy to see using Lemma \ref{lem:newlfs} that $\overline \cc'$ is as desired. 
\end{proof}

\begin{Remark}
The reader may wonder why we use only small brackets in the statement of Lemmas \ref{lem:bmb1} and \ref{lem:bmb2}, as opposed to including the big brackets from Lemma \ref{lem:newlfs}. Here the big brackets do seem natural, but later on when we consider multiple braid moves it seems best to convert to one type of bracket. 
\end{Remark}

\subsection{General bracketing rules}

Theorem \ref{th:brgeneral}, along with \S\ref{ss:brrank2}, gives a realization of $B(\infty)$ by bracketing rules for any simply braided reduced expression  $\ii$. Fix $i \in I$ and let $M_1,\dots,M_k$ be the nontrivial braid moves required to move $\alpha_i$ to the left of the convex order. Let $\Delta_1, \ldots, \Delta_k$ be the corresponding rank two sub-root-systems and, for each $1 \leq \ell \leq k$, let $\cc^\ell$ be the rank $2$ Lusztig datum defined by $c^\ell_{\alpha_i}=0$, and $c^\ell_\beta=c_\beta$ for all other $\beta \in \Delta_\ell$. 
Let $S_i^\ell$ be the string of brackets corresponding to $\cc^\ell$ as in \S\ref{ss:brrank2}.  Let
\begin{equation}
S_i(\cc) = S_i^k \cdots S_i^2 S_i^1  \ )^{c_{\alpha_i}}.
\end{equation}
Then $S_i(\cc)$ can be used to compute $f_i$.

\section{Examples: Calculating crystal operators using bracketing rules}
\label{sec:examples}

We now discuss a specific braidless reduced expression in each of the classical types, and the resulting crystal structures. All of these reduced expressions come from good enumerations, as in  \cite{Lit98}. There are actually many other good enumerations, and hence other braidless reduced expressions. It could be interesting to understand the combinatorics coming from these other enumerations.

In each of the following subsections we fix a particular reduced expression and use it to identify Lusztig data with Kostant partitions.  

\begin{Remark}
Below we often denote roots by stacks of numbers. For instance, 
\[
\stack{2\\3\,4\\2\\1}
\]
corresponds to the root $\alpha_1+2 \alpha_2+\alpha_3+\alpha_3$ in type $D_4$. 
The ordering of these stacks is chosen for the following reason: a root corresponds to a left bracket in $S_i$ exactly if one can place an $i$ at the top of the stack and still have either a root or a sum of roots. It corresponds to a right bracket in $S_i$ exactly if one can remove an $i$ from the top of the stack and still have either a root or a sum of roots.

\end{Remark}

\subsection{Type $A_n$}
The set of positive roots is $\{\alpha_{i,j}: 1 \le i \le j \le n\}$, where $\alpha_{i,j} = \alpha_i + \alpha_{i+1} + \cdots + \alpha_j$.  The word $\ii^A$ corresponding to the reduced expression 
\[
w_0 = (s_1s_2\cdots s_n)(s_1s_2\cdots s_{n-1}) \cdots (s_1s_2)s_1
\] 
is simply braided by Lemma \ref{lemma:reverse_Littelmann}. The corresponding order on positive roots $\alpha_{i,j} \prec \alpha_{i',j'}$ if and only if $i < i'$ or $i = i'$ and $j< j'$.  Then, given $i\in I$ and an $\ii^A$-Lusztig datum $\cc = (c_\beta)_{\beta\in\Phi^+}$, the string of brackets $S_i(\cc)$ is 
\[
\underbrace{)\cdots)}_{c_{\alpha_{1,i}}}\ 
\underbrace{(\cdots(}_{c_{\alpha_{1,i-1}}}\
\underbrace{)\cdots)}_{c_{\alpha_{2,i}}}\ 
\underbrace{(\cdots(}_{c_{\alpha_{2,i-1}}}\ \ \
\cdots\ \ \
\underbrace{)\cdots)}_{c_{\alpha_{i-i,i}}}\ 
\underbrace{(\cdots(}_{c_{\alpha_{i-i,i-1}}}\
\underbrace{)\cdots)}_{c_{\alpha_{i,i}}}.
\]

\begin{Example}\label{ex:pbw_A}
Consider type $A_3$ with $\ii = (1,2,3,1,2,1)$ and $i=2$.  The corresponding order on positive roots is $1 \prec \stack{2\\1} \prec \stack{3\\2\\1} \prec 2 \prec \stack{3\\2} \prec 3$. If $b \in \CB_\ii$ is such that $c^\ii(b) = (2,3,1,3,3,2)$, then the corresponding Kostant partition is 
\[
\arraycolsep=2pt
\begin{array}{cccccccccccccl}
1 & 1 & \stack{2\\1} & \stack{2\\1} & \stack{2\\1} & \stack{3\\2\\1} & 2 & 2 & 2 & \stack{3\\2} & \stack{3\\2} & \stack{3\\2}& 3 & 3.
\end{array}
\]
Placing the parts/roots in the order prescribed by Definition \ref{def:bracket_rule}, we get 

\begin{minipage}{4.72in}
\[
\arraycolsep=2pt
\begin{array}{ccccccccccccccllll}
& \stack{2\\1} & \stack{2\\1} & \stack{2\\1} & 1 & 1 & 2 & 2 & 2 \\[5pt]
S_2: \quad & ) & ) & ) & \tikzmark{leftA1}{{\color{red}(}} & {\color{red}(} & {\color{red})}  & \tikzmark{rightA1}{{\color{red})}} & )  &.
\end{array}
\]
\DrawLine[red, thick, opacity=0.5]{leftA1}{rightA1}
\end{minipage}

\noindent There are no uncanceled left brackets, so applying $f_2$ yields
\[
\arraycolsep=2pt
\begin{array}{ccccccccccccccl}
1 & 1 & \stack{2\\1} & \stack{2\\1} & \stack{2\\1} & \stack{3\\2\\1} & 2 & 2 & 2 &  {\color{blue}2} & \stack{3\\2} & \stack{3\\2} & \stack{3\\2}& 3 & 3.
\end{array}
\]
\end{Example}

\begin{table}[t]
\renewcommand{\arraystretch}{1.2}
\[
\begin{array}{cl}\toprule
\beta_{i,k}= \alpha_i + \cdots + \alpha_{k}, & 1\le i\le k \le n-1 \\
\gamma_{i,k}=\alpha_i + \cdots + \alpha_{n-2}+ \alpha_{n} + \alpha_{n-1}+ \cdots + \alpha_k, & 1\le i < k \le n\\ \midrule
\beta_{i,k}= \epsilon_i-\epsilon_{k+1}, & 1\le i\le k \le n-1 \\
\gamma_{i,k}= \epsilon_i+\epsilon_k, & 1\le i < k \le n\\\bottomrule
\end{array}
\]
\caption{Positive roots of type $D_n$, expressed both as a linear combination of simple roots and in the canonical realization following \cite{bourbaki}.}\label{posroots}
\end{table}

\subsection{Type $D_n$} \label{ss:tdn}
By Lemma \ref{lemma:reverse_Littelmann} the word $\ii^D$ associated to the reduced expression
\[
w_0 = 
(s_1s_2\cdots s_{n-1}s_ns_{n-2}\cdots s_2s_1)
\cdots 
(s_{n-2}s_{n-1}s_ns_{n-2})s_{n-1}s_n
\]  
is simply braided. Using notation from Table \ref{posroots},
the order of the positive roots corresponding to the subword $(i,i+1,\dots,n,n-2,\dots,i)$ of $\ii^D$, for $1\le i \le n-2$, is
\[
\begin{cases}
\beta_{i,i} \prec \beta_{i,i+1} \prec \dots \prec \beta_{i,n-2} \prec \beta_{i,n-1} \prec \gamma_{i,n} \prec \gamma_{i,n-1} \prec \cdots \prec \gamma_{i,i+1} & \text{ if } i \equiv 1 \bmod 2, \\
\beta_{i,i} \prec \beta_{i,i+1} \prec \dots \prec \beta_{i,n-2} \prec \gamma_{i,n} \prec \beta_{i,n-1} \prec \gamma_{i,n-1} \prec \cdots \prec \gamma_{i,i+1} & \text{ if } i\equiv 0 \bmod 2.
\end{cases}
\]
The ordering on the roots corresponding to the suffix $(n-1,n)$ of $\ii^D$ is 
\[
\begin{cases}
\beta_{n-1,n-1} \prec \gamma_{n-1,n} & \text{ if } n \equiv 0 \bmod 2, \\
\gamma_{n-1,n} \prec \beta_{n-1,n-1} & \text{ if } n \equiv 1 \bmod 2.
\end{cases}
\]
It follows that, for a Kostant partition $\cc = (c_\beta)_{\beta\in \Phi^+}$, the string of brackets $S_i^{\ii^D}(\cc)$ is
\[
\begin{cases}
\ \underbrace{)\cdots)}_{c_{\beta_{1,i}}}\ 
\underbrace{(\cdots(}_{c_{\beta_{1,i-1}}}\
\underbrace{)\cdots)}_{c_{\gamma_{1,i}}}\ 
\underbrace{(\cdots(}_{c_{\gamma_{1,i+1}}}\ 
\cdots\ 
\underbrace{)\cdots)}_{c_{\beta_{i-1,i}}}\ 
\underbrace{(\cdots(}_{c_{\beta_{i-1,i-1}}}\
\underbrace{)\cdots)}_{c_{\gamma_{i-1,i}}}\ 
\underbrace{(\cdots(}_{c_{\gamma_{i-1,i+1}}}\
\underbrace{)\cdots)}_{c_{\beta_{i,i}}}, 
& i \neq n, \\
\ \underbrace{)\cdots)}_{c_{\gamma_{1,n}}}\ 
\underbrace{(\cdots(}_{c_{\beta_{1,n-2}}}\
\underbrace{)\cdots)}_{c_{\gamma_{1,n-1}}}\ 
\underbrace{(\cdots(}_{c_{\beta_{1,n-1}}}\ 
\cdots\ 
\underbrace{)\cdots)}_{c_{\gamma_{n-2,n}}}\ 
\underbrace{(\cdots(}_{c_{\beta_{n-2,n-2}}}\
\underbrace{)\cdots)}_{c_{\gamma_{n-2,n-1}}}\ 
\underbrace{(\cdots(}_{c_{\beta_{n-2,n-1}}}\
\underbrace{)\cdots)}_{c_{\gamma_{n-1,n}}},
& i = n.
\end{cases}
\]

\begin{Example}
Consider the setup from Example \ref{ex:pbw_compute}.  The Kostant partition corresponding to that datum is
\[
\arraycolsep=2pt
\begin{array}{cccccccccccccccccccccl}
1 & 1 & \stack{2\\1} & \stack{3\\2\\1} & \stack{3\\2\\1} & \stack{3\\2\\1} & \stack{3\\2\\1} & \stack{4\\2\\1} & \stack{4\\2\\1} & \stack{3\,4\\2\\1} & \stack{2\\3\,4\\2\\1} & \stack{2\\3\,4\\2\\1} & \stack{2\\3\,4\\2\\1} & 2 & 2 & 2 & \stack{4\\2} & \stack{3\\2} & \stack{3\\2} & \stack{3\,4\\2} & 3 & 3.
\end{array}
\]
Arranging the parts/roots in the order prescribed by Definition \ref{def:bracket_rule}, we get
\[
\arraycolsep=3pt
\begin{array}{ccccccccccccccccccccclll}
&\stack{4\\2\\1} & \stack{4\\2\\1} & \stack{2\\1} & \stack{3\,4\\2\\1} & \stack{3\\2\\1} & \stack{3\\2\\1} & \stack{3\\2\\1} & \stack{3\\2\\1} & \stack{4\\2} & 2 & 2 & 2 & \stack{3\,4\\2} & \stack{3\\2} & \stack{3\\2} \\[5pt]
S_4:\quad & ) & ) & \tikzmark{leftD1}{\color{red}(} & \tikzmark{rightD1}{\color{red})} & {\color{blue}\boldsymbol{(}} & ( & ( & \tikzmark{leftD2}{\color{red}(} & \tikzmark{rightD2}{\color{red})} & ( & ( & \tikzmark{leftD3}{\color{red}(} & \tikzmark{rightD3}{\color{red})} & ( & ( & .
\end{array}
\]
\DrawLine[red, thick, opacity=0.5]{leftD1}{rightD1}
\DrawLine[red, thick, opacity=0.5]{leftD2}{rightD2}
\DrawLine[red, thick, opacity=0.5]{leftD3}{rightD3}

\noindent Hence applying $f_4$ gives
\[
\arraycolsep=2pt
\begin{array}{cccccccccccccccccccccl}
1 & 1 & \stack{2\\1} & \stack{3\\2\\1} & \stack{3\\2\\1} & \stack{3\\2\\1} & \stack{4\\2\\1} & \stack{4\\2\\1} & \stack{3\,4\\2\\1} & {\color{blue}\stack{3\,4\\2\\1}} & \stack{2\\3\,4\\2\\1} & \stack{2\\3\,4\\2\\1} & \stack{2\\3\,4\\2\\1} & 2 & 2 & 2 & \stack{4\\2} & \stack{3\\2} & \stack{3\\2} & \stack{3\,4\\2} & 3 & 3.
\end{array}
\]
\end{Example}

\subsection{Types $B_n$ and $C_n$}
These have the same Weyl group. By Lemma \ref{lemma:reverse_Littelmann}, the word
$\ii^{BC}$ associated to the reduced expression
\[
w_0 = 
(s_1s_2\dots s_{n-1}s_ns_{n-1} \cdots s_2s_1)
(s_2\cdots s_{n-1}s_ns_{n-1}\cdots s_2) 
\cdots 
(s_{n-1}s_ns_{n-1})
s_n
\] 
is simply braided.

\begin{table}[t]
\renewcommand{\arraystretch}{1.3}
\[
\begin{array}{cl}\toprule
\beta_{i,k}= \alpha_i + \cdots + \alpha_{k}, & 1\le i\le k \le n \\
\gamma_{i,k} =\alpha_i + \cdots + \alpha_{k-1}+2 \alpha_{k}+ \cdots + 2 \alpha_n , & 1\le i < k \le n \\\midrule
\beta_{i,k}= \epsilon_i-\epsilon_{k+1}, & 1\le i\le k \le n-1 \\
\beta_{i,n} = \epsilon_i, & 1\le i \le n\\
\gamma_{i,k} = \epsilon_i+\epsilon_k , & 1\le i < k \le n \\\bottomrule
\end{array}
\]
\caption{Positive roots of type $B_n$, expressed both as a linear combination of simple roots and in the canonical realization following \cite{bourbaki}.}\label{tab:Broots}
\end{table}

The positive roots for type $B_n$ are listed in Table \ref{tab:Broots}, and the convex order on $\Phi^+$ in type $B_n$ corresponding to the subword $(i,i+1,\dots,n-1,n,n-1,\dots,i+1,i)$ of $\ii^{BC}$ is 
\[
\beta_{i,i} \prec \beta_{i,i+1} \prec \cdots \prec \beta_{i,n} \prec \gamma_{i,n-1} \prec \gamma_{i,n-2} \prec \cdots \prec \gamma_{i,i+1}.
\]
Given a Kostant partition $\cc = (c_\beta)_{\beta\in\Phi^+}$, the string of brackets $S_i^{\ii^{BC}}(\cc)$ is 
\[
\begin{cases}
\ \underbrace{)\cdots)}_{c_{\beta_{1,i}}}\ 
\underbrace{(\cdots(}_{c_{\beta_{1,i-1}}}\
\underbrace{)\cdots)}_{c_{\gamma_{1,i}}}\ 
\underbrace{(\cdots(}_{c_{\gamma_{1,i+1}}}\ \ \
\cdots\ \ \
\underbrace{)\cdots)}_{c_{\beta_{i-1,i}}}\ 
\underbrace{(\cdots(}_{c_{\beta_{i-1,i-1}}}\
\underbrace{)\cdots)}_{c_{\gamma_{i-1,i}}}\ 
\underbrace{(\cdots(}_{c_{\gamma_{i-1,i+1}}}\
\underbrace{)\cdots)}_{c_{\beta_{i,i}}}, 
& i \neq n, \\
\ \underbrace{)\cdots)}_{c_{\beta_{1,n}}}\ 
\underbrace{(\cdots(}_{2c_{\beta_{1,n-1}}}\
\underbrace{)\cdots)}_{2c_{\gamma_{1,n}}}\ 
\underbrace{(\cdots(}_{c_{\beta_{1,n}}}\ \ \
\cdots\ \ \
\underbrace{)\cdots)}_{c_{\beta_{n-1,n}}}\ 
\underbrace{(\cdots(}_{2c_{\beta_{n-1,n-1}}}\
\underbrace{)\cdots)}_{2c_{\gamma_{n-1,n}}}\ 
\underbrace{(\cdots(}_{c_{\beta_{n-1,n}}}\
\underbrace{)\cdots)}_{c_{\beta_{n,n}}},
& i = n.
\end{cases}
\]

\begin{Example}
Consider the Kostant partition of type $B_3$ given by
\[
\arraycolsep=2pt
\begin{array}{cccccccccccccccccccccccl}
1 & 1 & 1 & 1 & \stack{2\\1} & \stack{3\\2\\1} & \stack{3\\2\\1} & \stack{3\\2\\1} & \stack{3\,3\\2\\1} & \stack{3\,3\\2\\1} & \stack{2\\3\,3\\2\\1} & \stack{3\\2} & \stack{3\\2} & \stack{3\\2} & \stack{3\\2} & \stack{3\\2} & \stack{3\,3\\2} & 3 & 3.
\end{array}
\]
We have 
\[
\arraycolsep=3pt
\begin{array}{ccccccccccl}
& \stack{2\\1} & 1 & 1 & 1 & 1 & \stack{2\\3\,3\\2\\1} & \stack{3\,3\\2\\1} & \stack{3\,3\\2\\1} \\[8pt]
S_2: &) & {\color{blue}\boldsymbol{(}} & ( & ( & \tikzmark{leftB11}{\color{red}(} & \tikzmark{rightB11}{\color{red})} & ( & \tikzmark{leftB1}{\color{red}(} & .
\end{array}
\]
\DrawLine[red, thick, opacity=0.5]{leftB11}{rightB11}

\noindent Hence, applying $f_2$ gives
\[
\arraycolsep=2pt
\begin{array}{cccccccccccccccccccccccl}
1 & 1 & 1 & {\color{blue}\stack{2\\1}} & \stack{2\\1} & \stack{3\\2\\1} & \stack{3\\2\\1} & \stack{3\\2\\1} & \stack{3\,3\\2\\1} & \stack{3\, 3\\2\\1} & \stack{2\\3\,3\\2\\1} & \stack{3\\2} & \stack{3\\2} & \stack{3\\2} & \stack{3\\2} & \stack{3\\2} & \stack{3\,3\\2} & 3 & 3.
\end{array}
\]
%\begin{minipage}{4.72in}
Similarly, 
\[
\arraycolsep=2pt
\begin{array}{cccccccccccccccccccccccl}
&\stack{3\\2\\1} & \stack{3\\2\\1} & \stack{3\\2\\1} & \stack{2\\1}  & \stack{3\,3\\2\\1} & \stack{3\,3\\2\\1}  & \stack{3\\2\\1} & \stack{3\\2\\1} & \stack{3\\2\\1} & \stack{3\\2} & \stack{3\\2} & \stack{3\\2} & \stack{3\\2} & \stack{3\\2} & \stack{3\,3\\2} & \stack{3\\2} & \stack{3\\2} & \stack{3\\2} & \stack{3\\2} & \stack{3\\2} & 3 & 3 \\[8pt]
S_3:&) & ) & ) & \tikzmark{leftB2}{\color{red}(}  {\color{red}(} & {\color{red})}  \tikzmark{rightB2}{\color{red})} & )  ) & \tikzmark{leftB3}{\color{red}(} & {\color{red}(} & {\color{red}(} & {\color{red})} & {\color{red})} & \tikzmark{rightB3}{\color{red})} & ) & ) & ) ) & {\color{blue}\boldsymbol{(}} & ( & ( & \tikzmark{leftB4}{\color{red}(} & {\color{red}(} & {\color{red})} & \tikzmark{rightB4}{\color{red})} & .\\[8pt]
\end{array}
\]
\DrawLine[red, thick, opacity=0.5]{leftB2}{rightB2}
\DrawLine[red, thick, opacity=0.5]{leftB3}{rightB3}
\DrawLine[red, thick, opacity=0.5]{leftB4}{rightB4}
%\end{minipage}

\noindent Hence applying $f_3$ gives
\[
\arraycolsep=2pt
\begin{array}{cccccccccccccccccccccccl}
1 & 1 & 1 & 1 & \stack{2\\1} & \stack{3\\2\\1} & \stack{3\\2\\1} & \stack{3\\2\\1} & \stack{3\,3\\2\\1} & \stack{3\,3\\2\\1} & \stack{2\\3\,3\\2\\1} & \stack{3\\2} & \stack{3\\2} & \stack{3\\2} & \stack{3\\2} & {\color{blue}\stack{3\,3\\2}} & \stack{3\,3\\2} & 3 & 3.
\end{array}
\]
\end{Example}

\begin{table}[t]
\renewcommand{\arraystretch}{1.3}
\[
\begin{array}{cl}\toprule
\beta_{i,k}= \alpha_i + \cdots + \alpha_{k}, & 1\le i\le k \le n-1 \\
\gamma_{i,k}= \alpha_i + \cdots + \alpha_{n-1}+\alpha_n + \alpha_{n-1} + \cdots + \alpha_k, & 1\le i \le k \le n \\\midrule
\beta_{i,k}= \epsilon_i-\epsilon_{k+1}, & 1\le i\le k \le n-1 \\
\gamma_{i,k}= \epsilon_i+\epsilon_k, & 1\le i \le  k \le n  \\\bottomrule
\end{array}
\]
\caption{Positive roots of type $C_n$, expressed both as a linear combination of simple roots and in the canonical realization following \cite{bourbaki}.}\label{tab:Croots}
\end{table}

The positive roots for type $C_n$ are listed in Table \ref{tab:Croots}. For $i=1,\dots,n-1$ the convex order corresponding to the subword $(i,i+1,\dots,n-1,n,n-1,\dots,i+1,i)$ of $\ii^{BC}$ is 
\[
\beta_{i,i} \prec
\beta_{i,i+1} \prec
\cdots
\beta_{i,n-1} \prec
\gamma_{i,i} \prec
\gamma_{i,n} \prec
\gamma_{i,n-1} \prec
\cdots
\prec
\gamma_{i,i+1}.
\]
Given a Kostant partition $\cc = (c_\beta)_{\beta\in\Phi^+}$ and $i \in I$, the string of brackets $S_i^{\ii^{BC}}(\cc)$ is 
\[
\begin{cases}\
\underbrace{)\cdots)}_{c_{\beta_{1,i}}}\ 
\underbrace{(\cdots(}_{c_{\beta_{1,i-1}}}\
\underbrace{)\cdots)}_{c_{\gamma_{1,i}}}\ 
\underbrace{(\cdots(}_{c_{\gamma_{1,i+1}}}\ \ \
\cdots\ \ \
\underbrace{)\cdots)}_{c_{\beta_{i-1,i}}}\ 
\underbrace{(\cdots(}_{c_{\beta_{i-1,i-1}}}\
\underbrace{)\cdots)}_{c_{\gamma_{i-1,i}}}\ 
\underbrace{(\cdots(}_{c_{\gamma_{i-1,i+1}}}\
\underbrace{)\cdots)}_{c_{\beta_{i,i}}}, 
& \text{ if } i\neq n, \\
\underbrace{)\cdots)}_{c_{\gamma_{1,1}}}\ 
\underbrace{(\cdots(}_{c_{\beta_{1,n-1}}}\
\underbrace{)\cdots)}_{c_{\gamma_{1,n}}}\ 
\underbrace{(\cdots(}_{c_{\gamma_{1,1}}}\ \ \ 
\cdots \ \ \ 
\underbrace{)\cdots)}_{c_{\gamma_{n-1,n-1}}}\ 
\underbrace{(\cdots(}_{c_{\beta_{n-1,n-1}}}\
\underbrace{)\cdots)}_{c_{\gamma_{n-1,n}}}\ 
\underbrace{(\cdots(}_{c_{\gamma_{n-1,n-1}}}\
\underbrace{)\cdots)}_{c_{\gamma_{n,n}}},
& \text{ if } i = n.
\end{cases}
\]

\begin{Example}
Let $\cc = (c_\beta)$ be the Kostant partition determined by the $\ii$-Lusztig datum from Example \ref{ex:C_PBW}; that is, the Kostant partition of type $C_3$ given by
\[
\arraycolsep=2pt
\begin{array}{cccccccccccccccccl}
1 & 1 & 1 & 1 & \stack{2\\1} & \stack{3\\2\,2\\1\,1} & \stack{3\\2\,2\\1\,1} & \stack{3\\2\,2\\1\,1} & \stack{3\\2\\1} & \stack{3\\2\\1} & \stack{3\\2\,2} & \stack{3\\2\,2} & \stack{3\\2\,2} & \stack{3\\2\,2} & \stack{3\\2\,2} & \stack{3\\2} & 3 & 3.
\end{array}
\]
To compute $f_3$, we have
\[
\arraycolsep=1pt
\begin{array}{cccccccccccccccccccccccl}
&\stack{3\\2\,2\\1\,1} & \stack{3\\2\,2\\1\,1} & \stack{3\\2\,2\\1\,1} & \stack{2\\1} & \stack{3\\2\\1} & \stack{3\\2\\1} & \stack{3\\2\,2\\1\,1} & \stack{3\\2\,2\\1\,1} & \stack{3\\2\,2\\1\,1} & \stack{3\\2\,2} & \stack{3\\2\,2} & \stack{3\\2\,2} & \stack{3\\2\,2} & \stack{3\\2\,2} & \stack{3\\2} & \stack{3\\2\,2} & \stack{3\\2\,2} & \stack{3\\2\,2} & \stack{3\\2\,2} & \stack{3\\2\,2} & 3 & 3 \\[8pt]
S_3:&) & ) & ) & \tikzmark{leftC0}{\color{red}(} & \tikzmark{rightC0}{\color{red})} & ) & \tikzmark{leftC1}{\color{red}(} & {\color{red}(} & {\color{red}(} & {\color{red})} & {\color{red})} & \tikzmark{rightC1}{\color{red})} & ) & ) & ) & {\color{blue}\boldsymbol{(}} & ( & ( & \tikzmark{leftC2}{\color{red}(} & {\color{red}(} & {\color{red})} & \tikzmark{rightC2}{\color{red})} & .
\end{array}
\]
\DrawLine[red, thick, opacity=0.5]{leftC0}{rightC0}
\DrawLine[red, thick, opacity=0.5]{leftC1}{rightC1}
\DrawLine[red, thick, opacity=0.5]{leftC2}{rightC2}

\noindent Hence the Kostant partitions corresponding to $f_3(\bf c)$ is 
\[
\arraycolsep=2pt
\begin{array}{ccccccccccccccccccl}
1 & 1 & 1 & 1 & \stack{2\\1} & \stack{3\\2\,2\\1\,1} & \stack{3\\2\,2\\1\,1} & \stack{3\\2\,2\\1\,1} & \stack{3\\2\\1} & \stack{3\\2\\1} & \stack{3\\2\,2} & \stack{3\\2\,2} & \stack{3\\2\,2} & \stack{3\\2\,2} & \stack{3\\2} & {\color{blue}\stack{3\\2}} & {\color{blue}\stack{3\\2}} & 3 & 3.
\end{array}
\]
Notice that, in accordance with Lemma \ref{lem:bmb2}, $c_{\alpha_2+\alpha_3}$ has increased by 2. 
\end{Example}

\begin{acknowledgements}
We thank the developers of Sage \cite{sage,combinat}, where many of the motivating calculations for this work were executed. We also thank Dinakar Muthiah for useful discussions. B.S.\ thanks Loyola University of Chicago for its hospitality during a visit in which this manuscript began.
\end{acknowledgements}

\bibliography{KP-crystal}{}
\bibliographystyle{amsplain}

\end{document}